\documentclass[10pt,oneside,shortlabels]{amsart}
\usepackage{amsaddr}
\usepackage[T1]{fontenc}
\usepackage[utf8]{inputenc}
\usepackage{amsthm,amsmath}
\usepackage{amsbsy}
\usepackage{amstext}
\usepackage{amssymb}
\usepackage{esint}
\usepackage{bbm,bm}
\usepackage{xcolor}
\usepackage{comment}
\makeatletter
\numberwithin{equation}{section}
\numberwithin{figure}{section}

\usepackage{mathtools}

\newtheorem{theorem}{Theorem}[section]
\newtheorem*{thmA}{Theorem~A}
\newtheorem*{thmBB}{Theorem~B}
\newtheorem*{thmC}{Theorem~C}
\theoremstyle{definition}

\newtheorem {prop} [theorem] {Proposition}
\newtheorem {cory} [theorem] {Corollary}
\newtheorem {conjecture} [theorem] {Conjecture}
\newtheorem {lem} [theorem] {Lemma}
\newtheorem {defi} [theorem] {Definition}
\newtheorem {obs} [theorem] {Remark}

\newcommand{\R}{\ensuremath{\mathbb{R}}}

\def\sideremark#1{\ifvmode\leavevmode\fi\vadjust{\vbox to0pt{\vss 
    \hbox to 0pt{\hskip\hsize\hskip1em           
 \vbox{\hsize2cm\tiny\raggedright\pretolerance10000
 \noindent #1\hfill}\hss}\vbox to8pt{\vfil}\vss}}}%

\usepackage[all]{xy}        
\newcommand{\xyL}[1]{%
	\xydef@\xymatrixrowsep@{#1}
} 
\newcommand{\xyC}[1]{%
	\xydef@\xymatrixcolsep@{#1}
} 
\begin{document}
\title{Reading multiplicity in unfoldings\\ from $\varepsilon$-neighborhoods of orbits}

\author[R. Huzak]{Renato Huzak}
\address{Hasselt University, Campus Diepenbeek, Agoralaan Gebouw D, Diepenbeek 3590, Belgium}
\email{{renato.huzak@uhasselt.be}}

\author[P. Marde\v{s}i\'c]{Pavao Marde\v{s}i\'c}
\address{Universit\'e de Bourgogne, Institut de Mathématiques de Bourgogne, 9 avenue Alain Savary, 21078 Dijon, France.
and University of Zagreb, Faculty of Science, Department of Mathematics, Bijeni\v cka cesta 30 / Horvatovac 102a, 10000 Zagreb, Croatia.
}
\email{{pavao.mardesic@u-bourgogne.fr}}

\author[M. Resman]{Maja Resman}
\address{University of Zagreb, Faculty of Science, Department of Mathematics, Bijeni\v cka cesta 30 / Horvatovac 102a, 10000 Zagreb, Croatia.}
\email{{maja.resman@math.hr}}

\author[V. \v{Z}upanovi\'c]{Vesna \v{Z}upanovi\'c}
\address{University of Zagreb, Faculty of Electrical Engineering and Computing, Unska 3, 10000 Zagreb, Croatia.}
\email{{vesna.zupanovic@fer.unizg.hr}}

\date{July 5, 2023}



\thanks{Corresponding author: Maja Resman, maja.resman@math.hr}

\begin{abstract} We consider generic  analytic $1$-parameter unfoldings of saddle-node 
germs of analytic vector fields on the real line, their time-one maps and the Lebesgue measure of $\varepsilon$-neighborhoods of the orbits of these time-one maps. 

The box dimension of an orbit gives the asymptotics of the principal term of this Lebesgue measure and it is known that it is discontinuous at bifurcation parameters. 

In order to recover continuous dependence of the asymptotics on the parameter, here we expand asymptotically the Lebesgue measure of $\varepsilon$-neighborhoods of orbits of time-one maps in a \emph{Chebyshev} system, \emph{uniformly with respect to the bifurcation parameter}. We use \emph{\' Ecalle-Roussarie-type compensators}. 

We show how the number of 
fixed points of the time-one map born in the universal analytic unfolding of the parabolic point 
 corresponds to the number of  terms vanishing in
this uniform expansion of the Lebesgue measure of $\varepsilon$-neighborhoods of orbits.

%
\end{abstract} 

\maketitle
\tableofcontents

\section{Introduction}

In this article we study diffeomorphisms  $f(x)$, or rather families of diffeomorphisms $f_\nu(x)$ on the real line. 
We want to relate the \emph{multiplicity} of a fixed point of a diffeomorphism $f$ with the local dynamical properties of the \emph{density} of its orbit.

By the \emph{multiplicity} of a fixed point of a diffeomorphism $f$, in a family $f_\nu$, we mean the maximal number of fixed points born from the fixed point in the family of diffeomorphisms $f_\nu$ deforming $f$.

The dynamical \emph{density} properties of an orbit   are encoded by the \emph{tube function} $\ell(\varepsilon)=\ell(T_{\varepsilon,\nu})$, defined below using the Lebesgue measure $\ell$ of $\varepsilon$-neighborhoods of  orbits converging to a fixed point for the family of diffeomorphisms $f_\nu$. 

The idea of reading the multiplicity of a fixed point of a diffeomorphism from tube functions of its orbits was already explored in \cite{lana, mrz, hopf}. However, no uniform expansion of the tube function for the family $f_\nu$ was given. It is known that, for studying bifurcations of zeros, expansions have to be uniform with respect to the parameters. The  problem of constructing a uniform expansion of the tube function, and relating its bifurcation properties to the bifurcations of fixed points is addressed in the present paper.

Here, for simplicity, we use only the tail part $T_{\varepsilon,\nu}$ of the $\varepsilon$-neighborhood of an orbit (see \eqref{eq:TN}), as it carries the same information as the full tube function.  
Given a diffeomorphism $f$ defined in a real interval, we associate with it the \emph{displacement function} $g=id-f$, where $id$ is the identity function. 

Recall that for differentiable functions,  the multiplicity of a fixed point in a sufficiently general unfolding is given by the number of terms vanishing in the Taylor expansion at the fixed point.
In particular, this is the case for the displacement function $x\mapsto g(x)$,

On the other hand, the tube function $\varepsilon\mapsto  \ell(T_{\varepsilon,\nu})$ of an unfolding $f_\nu$ of $f$ is not differentiable at $\varepsilon=0$.
If a family of real functions is not necessarily differentiable, but admits an asymptotic expansion in a Chebyshev scale (see Definition \ref{def:CH}), uniform with respect to the parameter $\nu$, then the multiplicity of a zero is given by the number of terms vanishing in this asymptotic expansion. 

We show in Theorem B
that the tube function $\ell(T_{\varepsilon,\nu})$, in the case studied in this paper, admits a uniform asymptotic expansion in a Chebyshev scale, multiplied by a common function $I$, see \eqref{eq:Cheb}. However, by its definition as a Lebesgue measure of a set, $\ell(T_{\varepsilon,\nu})$ cannot be zero for any $\varepsilon>0$: the function $I(\varepsilon,\nu)$ explodes precisely in points where $\frac{\ell(T_{\varepsilon,\nu})}{I(\varepsilon,\nu)}$ vanishes, thus canceling the zeros. 

By abuse, by multiplicity in such a family, we will mean the number of leading terms vanishing in the asymptotic scale.

The motivation for this article lies in the conjecture that we formulate here:

\begin{conjecture}
 \label{conj} Let $f$ be a local diffemorphism of the real line at a fixed point $x=0$, which is attractive from one side.
 Let $f_\nu$, $\nu\in \mathbb{R}^k$,  be a generic deformation of the diffeomorphism $f$  i.e. a deformation realizing the maximal multiplicity of the fixed point.
 Consider the restriction of $f_\nu$, to $\nu\in V\subset\mathbb{R}^k$, such that for all $\nu\in V$, $f_\nu$
 has a fixed point born from $x=0$.
 
Consider a point $x_0$ whose orbit by $f_0$ converges to the fixed point $x=0$. Let $\ell(T_{\nu,\varepsilon})$
be the corresponding tube function of $f_\nu$. 
Then, there exists a function $I(\varepsilon,\nu)$ and an asymptotic expansion of the length of the tail $\ell(T_{\nu,\varepsilon})$ 
uniform in $\nu\in V$, as $\nu\to 0$, see \eqref{eq:expa1}, in a Chebyshev scale multiplied termwise by $I(\nu,\varepsilon)$.

The multiplicity of a fixed point of a diffeomorphism $x\mapsto f(x)$ in the family $f_\nu$, $\nu\in V$,
coincides with the number of vanishing terms of
$\ell(T_{\nu,\varepsilon})$ at parameter value $\nu=0$ in this asymptotic scale.
\end{conjecture}
\medskip
Note that the two function: $g_\nu(x)$ and $\ell(T_{\nu,\varepsilon})$, related by the conjecture, 
live in completely different spaces: the phase space of values of $x$, for $f_\nu-\mathrm{id},$ and the parameter  $\varepsilon$ measuring the size of the $\varepsilon$-neighborhoods of orbits for the tube function $\ell(T_{\varepsilon,\nu})$.

Our principal result (Theorem A) is the proof of the Conjecture \ref{conj} in the simplest non-trivial case
of generic saddle-node bifurcations.
That is,  diffeomorphisms having a parabolic fixed point bifurcating into a diffeomorphism with two hyperbolic fixed points on the real line.
Moreover, we assume that this parabolic diffeomorphism is given as time-one map of a vector field.

\medskip

\color{black}

\medskip


\subsection{Outline of the results}

 Consider generic  analytic $1$-parameter  unfoldings of real  analytic saddle-node vector fields $X_\nu$, whose flow is given by $$\frac{dx}{dt}=F(x,\nu), \ x\in\R, \ \nu\in\mathbb R,\ (x,\nu)\to (0,0).$$
Then, two hyperbolic singular points are born from the saddle-node singular point at $\nu=0$  in the so-called \emph{saddle-node bifurcation} (see e.g. \cite{guckholm}). By $f_\nu$, we denote the time-one map of $X_\nu$, which is an unfolding of a parabolic diffeomorphism. Let $g_\nu:=\mathrm{id}-f_\nu$ be the corresponding displacement function. 

In the main Theorem~A (Section 1.1), we give a $1-1$ correspondence between the asymptotic expansion of the time-one map $f_\nu$ (i.e. of the displacement function $g_\nu$) in the phase space and the uniform asymptotic expansion of the tail tube function of its orbit in an appropriate compensator variable. Recall the reason for introducing compensators: in the $\varepsilon$-space the expansion of the tube function was not uniform with respect to the parameter $\nu$. Theorem A shows that we can read the multiplicity of a fixed point in the given bifurcation (how many fixed points can bifurcate from it in the given family), from the  number of terms vanishing in the uniform asymptotic expansion of the Lebesgue measure of the tail of an orbit of the time-one map. 

Proposition~\ref{prop:pm} and Theorem~B give precisely these uniform asymptotic expansions of the tail tube functions $\ell(\varepsilon)$ in appropriate Chebyshev systems including compensators. They are used in the proof of the main Theorem A at the end of Section~\ref{sec:main}. Proposition~\ref{prop:pm} concerns the model vector field case, while Theorem~B is for general vector fields of the form \eqref{eq:polje} under generic assumptions. Note that, due to computational reasons, the expansions are given for the continuous counterpart $\ell^c(T_{\varepsilon,\nu})$ of the length first and then, in Corollary~\ref{cor:discrete}, for the standard length $\ell(T_{\varepsilon,\nu})$. For the standard length $\ell(T_{\varepsilon,\nu})$, some additional oscillatory terms appear in the expansion due to the step function nature of the discrete critical time separating the tail and the nucleus, see Subsection~\ref{sec:conttime}.

In Theorem~C the expansions from Proposition~\ref{prop:pm} and Theorem~B are regrouped so that (in general infinitely many) terms from the same group at $\nu\neq 0$ merge to the same asymptotic term at the bifurcation value $\nu=0$. This illustrates the earlier mentioned phenomenon that the limit as $\nu\to 0$ does not commute with asymptotic expansion in $\varepsilon$ of the tube function. Also, we get in Section~\ref{ss:prvi} the asymptotic expansions in $\varepsilon\to 0$  of $\ell(T_{\varepsilon,\nu})$ in the case $\nu>0$ and $\nu=0$. The expansions have qualitatively different terms, which, in particular, results in a jump in the box dimension at the moment of bifurcation. Theorem~C is alsoused in Remark~\ref{rem:zadnji} for reading the formal class of the unfolding from fractal data.

\color{black}


\medskip


Finally, note that the multiplicity (here $2$) is one of the analytic invariants of saddle-node vector field unfoldings \cite{kostov, KR}.  For reading the other analytic invariants by tube function, see Remark \ref{rem:zadnji}.
We recall also that in \cite{kmrr} parabolic germs were studied (not depending on parameters and not necessarily given by a time-one map of a saddle-node). It was shown how to read the analytic invariants of these germs from their orbit.

\subsection{Main results}\label{sub:mr}
We consider an analytic germ of a system
\begin{equation}\label{eq:polje}
\frac{dx}{dt}=F(x,\nu),
\end{equation}
with $F$ real, analytic germ in $x$ and in parameter $\nu$, and with a non-hyperbolic singular point $x=0$ at the bifurcation value $\nu=0$ (i.e. $F(0,0)=0$, $F_x(0,0)=0$), satisfying the generic assumptions:
\begin{align}\label{eq:gene}
F_{\nu}(0,0)\neq 0,\ F_{xx}(0,0)\neq 0.    
\end{align}
Under these assumptions, the saddle-node point at $x=0$ bifurcates at $\nu=0$  into two  hyperbolic points on the real axis: one attracting and one repelling,  for $\nu\in(0,\delta)$, or $\nu\in(-\delta,0)$ 
depending on the sign of $F_\nu$ and $F_{xx}$ in \eqref{eq:gene}. 
For details, see Section~\ref{sec:snb}.

\begin{defi}[tube function $\ell(T_{\varepsilon,\nu})$]\label{l}
Take $x_0$ in the attracting basin of the saddle-node point, sufficiently close to $0$. Consider the time-one map $f_\nu$ of \eqref{eq:polje} and its orbit with initial point $x_0$:
$$\mathcal O_{f_\nu}(x_0):=\{f_\nu^{n}(x_0):\ n\in\mathbb N_0\}.$$ By $\mathcal O_{f_\nu}(x_0)_\varepsilon$, we denote its $\varepsilon$-neighborhood, $\varepsilon>0$. By \cite{tricot},
the $\varepsilon$-neighborhood of the orbit $\mathcal O_{f_\nu}(x_0)$ is a disjoint union of two parts, the tail and the nucleus:
\begin{equation}\label{eq:TN}
 \mathcal O_{f_\nu}(x_0)_\varepsilon=T_{\varepsilon,\nu}\cup N_{\varepsilon,\nu}.
\end{equation}
Here, the \emph{tail} $T_{\varepsilon,\nu}$ is the union of finitely many disjoint intervals, and the  \emph{nucleus} $N_{\varepsilon,\nu}$ one interval made of a  union of infinitely many overlapping $\varepsilon$-neighborhoods of points in the orbit. By $\mathcal\ell$ we denote the Lebesgue measure. We call the function 
$\ell(T_{\varepsilon,\nu})$, the \emph{tube function of the family $f_\nu$}. 
\end{defi}

In \cite{mrz,formal}, we studied the length $ \ell(\mathcal O_{f_\nu}(x_0)_\varepsilon)=\ell (T_{\varepsilon,\nu})+ \ell(N_{\varepsilon,\nu})$.
However, the essential information is carried already by $\ell (T_{\varepsilon,\nu})$ \cite{frankenhuijsen,kmrr}. Hence, we investigate here only this term. Moreover, the dependence on $x_0$ is not essential.

\bigskip

In the main Theorem A below, we use the notion of Chebyshev systems. Chebyshev systems are generalizations of Taylor power monomial scales, on which the \emph{division-derivation algorithm} can be performed, see \cite{mardesic}. More precisely,  
\begin{defi}[Chebyshev system, \cite{mardesic}]\label{def:CH} Let $I=(-r,r)$ or $I=[0,r)$, $r>0$. A finite sequence $\{u_0= 1, u_1, u_2,\ldots\,u_m\}$, $m\in\mathbb N$, of
functions of the class $C(I)\cap C^m(I\setminus\{0\})$,  is called a Chebyshev system if the functions $D_i(u_k)$, $i,\,k = 0,\ldots,m,$ are well defined on $I$ inductively by the following division and differentiation algorithm:
\begin{align}
&D_0(u_k) = u_k,\label{eq:CH}\\
&D_{i+1}(u_k) = \frac{(D_i(u_k))'}{(D_i(u_{i+1}))'},\ i = 0,\ldots,m-1,\nonumber
\end{align}
for every $0\leq k\leq m$, except possibly at $x = 0$, to which they are extended by continuity.
\end{defi}
\noindent We allow the functions $u_i$ to depend continuously on a parameter $\nu$, but we require that the differential operators $D_i$ from \eqref{eq:CH} be well-defined on an interval $I$ uniform with respect to $\nu$.

\noindent Note that by linearity \eqref{eq:CH} extends to operators well-defined on the space of functions generated by $\{u_0= 1,\ldots,u_m\}$. 
\medskip
    
Let the unfolding  $g_{\nu}$ on $I$ be generated by $\{u_0,\ldots,u_m\}$ and let $g_0=a_k u_k+\ldots+a_m u_m$, $k\geq 0$, $a_k\neq 0$. Assume moreover that there exists $r>0$ chosen uniformly with respect to $\nu$ such that $u_0$ has no zeros in $I$, except possibly the origin. If $u_0(0)\neq 0$ then, by Rolle's theorem, the maximal number of zeros that can be born from the origin in unfolding $g_\nu$ is bounded above by $k$, i.e. by the number of missing Chebyshev terms in $g_0$ before the leading one. If $u_0(0)=0$ then, adding the zero at the origin, the maximal number of zeros that can be born from $0$ in unfolding $g_\nu$ is bounded above by $k+1$.
 The latter will be the case with our expansions both of the displacement and of the tube function in Theorem A.

\color{black}
\smallskip

We now state the main theorem.
\begin{thmA}\label{thmB} Let \eqref{eq:polje} be a  generic $1$-parameter analytic unfolding of a parabolic fixed point, satisfying the generic assumptions \eqref{eq:gene}.  Assume unilateral values of parameter $\nu$ for which parabolic point unfolds into two hyperbolic points: without loss of generality, $\nu\in[0,\delta)$. Let $f_\nu$ be the time-one map and $\mathcal O_{f_\nu}(x_0)$ be its attractive orbit starting at $x_0$. 

There exists a compensator variable $\eta(\varepsilon,\nu)$ and an asymptotic expansion of the length of the tail $\ell(T_{\nu,\varepsilon})$ in a Chebyshev system,  uniform in $\nu\in [0,d)$, as $\eta\to 0$, see \eqref{eq:expa1}. 

There is a $1-1$ correspondence between the expansion of the length $\ell(T_{\nu,\varepsilon})$ in the $\eta$ variable and the Taylor expansion of the \emph{displacement function} $g_\nu:=\mathrm{id}-f_\nu$ in the phase $x$-variable, in the following sense. For every value of the parameter $\nu$, the number of vanishing terms of the expansions of $\ell(T_{\nu,\varepsilon})$ and $g_\nu$, in the corresponding systems, is the same  $($here, for $1$-parameter unfoldings, at most $1)$.
\end{thmA}

Recall that the number of terms in a Chebyshev expansion of a family of functions vanishing at the bifurcation value gives the multiplicity. Zero points of the displacement function $g_\nu$ correspond to the fixed points of the time-one map, that is, to the singularities of the vector field.

As a direct consequence of Theorem~A, the multiplicity in the bifurcation can be read from the expansion of tube function $\ell(T_{\varepsilon,\nu})$ in compensator variable $\eta$. Precisely, by expansions given in Proposition~\ref{prop:pm} Theorem B, the multiplicity is the number of vanishing Chebyshev terms of the tube function at bifurcation value $\nu=0$ plus $1$. Here it equals $2$. 

\begin{obs}\label{rem:I}
Note that the number of vanishing terms in the uniform asymptotic scale for the tube function  (see Proposition~\ref{prop:pm} and Theorem~B) at the bifurcation value $\nu=0$, which we call \emph{multiplicity} of the expansion of the tube function in the given asymptotic scale, does not imply the number of zero points of the tube function that bifurcate from $\varepsilon=0$. Indeed, the tube function $\ell(T_{\varepsilon,\nu})$ measures the length of the $\varepsilon$-neighborhood of the orbit and is therefore strictly positive for all $\varepsilon>0$ and zero only at $\varepsilon=0$. The reason for that is the term $I(\nu,\eta)$ which is common to all terms of the scale for the expansion, which is strictly positive but explodes to $\infty$ exactly at zero points of the quotient $\frac{\ell(T_{\varepsilon,\nu})}{I(\nu,\eta)}$. 

Despite these singularities of the strictly positive common factor $I(\nu,\eta)$, by abuse we will nevertheless call the asymptotic scale for $\ell(T_{\varepsilon,\nu})$ in Proposition~\ref{prop:pm} and Theorem~B \emph{Chebyshev}. The number of vanishing terms of the expansion of the tube function at $\nu=0$ gives only an upper bound on the number of zero points of $\ell(T_{\varepsilon,\nu})$, none of which are really zero points of the tube function, due to singularities of the common factor. However, this upper bound equals the multiplicity of the  generic unfolding.
\end{obs}

\color{black}

\section{Notations and main objects}
\subsection{Normal forms for the saddle-node bifurcation}\label{sec:snb}

The saddle-node bifurcation is a generic $1$-parameter bifurcation of $1$-dimensional vector fields. 

Consider a system 
\begin{equation}\label{eq:F} 
\frac{dx}{dt}=F(x,\nu),
\end{equation}
with 
$F$ real and analytic in $x$ and in $\nu$, having at $\nu=0$ a \emph{non-hyperbolic}\footnote{\emph{Non-hyperbolic} means that $F_x(0,0)=0$.} singular point $x=0$ (i.e. undergoing a bifurcation of the singular point at $\nu=0$) and which satisfies generic assumptions $F_{xx}(0,0)\neq 0$ and $F_\nu(0,0)\neq 0$.




\medskip

By \cite{kostov,KR}, by a \emph{weak analytic change of variables}, the system 
\eqref{eq:F} can be brought to the form

\begin{equation}\label{eq:model}
\frac{dx}{dt}=F_{mod}(x,\nu),\ F_{mod}(x,\nu):=\frac{-x^2+\nu}{1+\rho(\nu)x},\ \nu\in(-\delta,\delta).
\end{equation}

\noindent The two analytic invariants are the multiplicity $k=2$ and the \emph{residual invariant} $\nu\mapsto \rho(\nu),\ \nu\geq 0.$ Here, weak analytic 
change of variables is
an analytic germ (germified at $x=0,\ \nu=0$) of change of variables, fibered in $\nu$ \cite{MRouss, KR}, i.e. of the form
\begin{equation}\label{eq:tmaps}
\Phi(x,\nu)=(\varphi_{\nu}(x),h(\nu)),
\end{equation}
where $h$ is an analytic diffeomorphism such that $h(0)=0$ and $\varphi_\nu\in \mathbb R\{x\}$, $\nu\in(-\delta,\delta)$, is an analytic diffeomorphism 
conjugating one to the other. For the corresponding time-one maps it holds that:
$$
f_{h(\nu)}^{mod}=\varphi_{\nu}\circ f_{\nu}\circ\varphi_{\nu}^{-1}.   
$$
Note that the adjective \emph{weak} refers to the possible bijective reparametrization $h(\nu)$ of the parameter $\nu$. 

For $\nu<0$ there are no real singular points of the model. For $\nu>0$, there are two singular hyperbolic points: attracting $\sqrt{\nu}$ and repelling $-\sqrt{\nu}$. For $\nu=0$ the point zero is a  saddle-node singular point, attracting from the right and repelling from the left. Choose an initial point $x_0\neq 0$. For small values of $\nu\in[0,\delta)$, $\delta>0$, $x_0$ lies outside $[-\sqrt{\nu},\sqrt{\nu}]$, but stays in the attracting resp. repelling basin of $\sqrt\nu$ resp. $-\sqrt\nu$,  \cite{MRouss}. 

We suppose $x_0>0$, so that (for $\nu$ sufficiently small and positive) it lies in the attractive basin for the bifurcation. Otherwise, if we choose $x_0<0$, we consider the opposite field, $\dot x=x^2-\nu$, 
so that $x_0$ lies again in the basin of attraction. 



\subsection{The continuous-time length of the tail of orbits}\label{sec:conttime}

In this subsection we show how to calculate the length $\ell(T_{\varepsilon,\nu})$ of the tail of the orbit $\mathcal O_{f_\nu}(x_0)$.  Recall \cite{tricot}, that 
the  \emph{discrete critical time} $n_\varepsilon^\nu\in\mathbb N$,  
 separates the order of iterates of $f_\nu$, for which   $\varepsilon$-neighborhoods of points in the orbit are not overlapping, from the order for which  overlapping of these $\varepsilon$-neighborhoods starts.

It is determined by the inequalities:
$$
f_\nu^{n_\varepsilon^\nu}(x_0)-f_\nu^{n_\varepsilon^\nu+1}(x_0)\leq 2\varepsilon,\ \ f_\nu^{n_\varepsilon^\nu-1}(x_0)-f_\nu^{n_\varepsilon^\nu}(x_0)> 2\varepsilon.
$$
The $\varepsilon$-neighborhood of an orbit $\mathcal{O}_{f_\nu}(x_0)$
consists of the nucleus $N_{\varepsilon,\nu}$ and the tail $T_{\varepsilon,\nu}$.
The nucleus is the overlapping part of $\varepsilon$-neighborhood of the orbit and the 
tail consists of $n_\varepsilon^\nu$ nonintersecting  intervals each of length $2\varepsilon$. Hence 

$$
\ell(T_{\varepsilon,\nu})=n_{\varepsilon}^\nu\cdot  2\varepsilon.
$$

Since the critical time $\varepsilon\mapsto n_\varepsilon^\nu$ is a step function, for a fixed $\nu$, the function $\varepsilon\mapsto\ell(T_{\varepsilon,\nu})$ does not have a full asymptotic expansion as $\varepsilon\to 0$. Therefore, we replace $n_\varepsilon^\nu$ by the so-called \emph{continuous critical time } $\tau_\varepsilon^\nu\in\mathbb R$ satisfying:
$$
f_\nu^{ \tau_\varepsilon^\nu}(x_0)-f_\nu^{\tau_\varepsilon^\nu+1}(x_0)=g_\nu(f_\nu^{ \tau_\varepsilon^\nu}(x_0))=\,2\varepsilon.
$$
Here, $\{f_\nu^t:t\in\mathbb R\}$ is the \emph{flow} of the field $X_\nu=F(x,\nu)\frac{d}{dx}$ given by \eqref{eq:F}. 

Note that $f_\nu:=f_\nu^1$ is the (germ of) time-one map of $X_\nu$. The continuous critical time $\tau_\varepsilon^\nu$ can be understood  as the time needed to move along the field from the initial point $x_0$ to the point $x$ whose displacement function value $g_\nu(x)$ is exactly equal to $2\varepsilon$, for every $\varepsilon>0$. Note that, as $\varepsilon$ tends to $0$, we choose the \emph{positive} time $\tau_\varepsilon^\nu$ (tending to $+\infty$), such that the flow $f_\nu^{\tau_\varepsilon^\nu}(x_0)$ approaches the attracting singular point $\sqrt\nu$ from the side of the initial point $x_0$. That is, although $g_\nu$ is multivalued around the singular point, we take the inverse of the unilateral, strictly increasing restriction of $g_\nu$ containing $x_0$.

We now define the \emph{continuous-time length of $T_{\varepsilon,\nu}$} (see \cite{mrrz2}) by:
\begin{equation}\label{eq:prva}
\ell^c(T_{\varepsilon,\nu}):=\tau_\varepsilon^\nu\cdot 2\varepsilon.
\end{equation}
Equivalently, 
\begin{align}\label{eq:druga}
\ell^c(T_{\varepsilon,\nu})=&(\Psi_\nu(g_\nu^{-1}(2\varepsilon))-\Psi_\nu(x_0))\cdot 2\varepsilon,
\end{align}
where the \emph{time coordinate germ} $\Psi_\nu$, defined up to an additive constant, is the  trivialization coordinate for the flow of field $X_\nu=F(x,\nu)\frac{d}{dx}$ from \eqref{eq:F}, satisfying
$$
\Psi_{\nu}(f_\nu^t(x_0))-\Psi_\nu(x_0)=t,\ t\in\mathbb R,
$$ 
or, equivalently, 
$$
\Psi_\nu'(x)=\frac{1}{F(x,\nu)}.
$$
For details of the definition and the relation between the two definitions, see \cite{mrrz2}.

\medskip

\color{red}\color{black}
We consider only the case when $\nu\geq 0$. In the case that $\nu<0$, the fixed points lie on the imaginary axis, and are not hyperbolic but indifferent, that is, their linear part is a rotation. On the real line the vector field passes from $-\infty$ to $+\infty$ in a finite time, and there are no singular points on the real line. This case will be a subject of future research, see Section~\ref{ss:drugi}. 


\smallskip 

All our theoretical results in the following sections will first be given for the continuous length $\ell^c(T_{\varepsilon,\nu})$. However, from the orbit it is natural to \emph{read} the standard length $\ell(T_{\varepsilon,\nu})$, as the sum of the lengths of $2\varepsilon$-intervals centered at points of the orbit of the time-one map $f_\nu$ before they start overlapping. By \cite{frenki, resman}, for a fixed $\nu$, this function does not allow the full asymptotic expansion in $\varepsilon\to 0$, due to oscillatory terms. Nevertheless, by Corollary~\ref{cor:discrete}, the Chebyshev system given in Proposition~\ref{prop:pm} and Theorem B can easily be adapted for the standard length $\ell(T_{\varepsilon,\nu})$.


\section{Uniform asymptotic expansions \\ of the length functions $\ell(T_{\varepsilon,\nu})$ and $\ell^c(T_{\varepsilon,\nu})$}\label{sec:main}
The Subsections~\ref{subsec:mod} (Proposition~\ref{prop:pm}) and \ref{subsec:nonmod} (Theorem B and Corollary~\ref{cor:discrete}) respectively  give the Chebyshev systems for $\ell(T_{\varepsilon,\nu})$ and $\ell^c(T_{\varepsilon,\nu})$ for the model family and for generic saddle-node families respectively. The model case is used in the proof of the general case by performing an  analytic change of variables.

\subsection{Compensators}
In the sequel we first define three compensators that we use in the uniform expansions in Proposition~\ref{prop:pm}. We use the name \emph{compensator} for  {elementary} expressions in variable $x$ and parameter $\nu$, i.e. expressions that cannot be further asymptotically {expanded} uniformly in $\nu$.
\begin{defi}\label{def:rus} \cite{roussarie} Let $\nu$ and $x$ be small. The function
$$
\omega(x,\nu):=\frac{x^{-\nu}-1}{\nu}
$$
is called the \emph{\' Ecalle-Roussarie compensator}. 
\end{defi}
Note that, pointwise, $\omega(x,\nu)\to-\log x$, as $\nu \to 0$. The convergence becomes uniform in $x$, if we multiply by $x^\delta$, $\delta>0$.

\begin{defi}[The inverse compensator]\label{def:lc} For $x>0$ and $\nu\in(-\delta,\delta)$, we call the function
$$
\alpha(x,\nu):=\frac{1}{\nu}\log\left(1+\frac{\nu}{x}\right).
$$
 the \emph{inverse compensator}.
\end{defi}
\noindent The name comes from Definition~\ref{def:rus} and the fact that $$\alpha(x,\nu)=-\log\circ\, \omega^{-1}\left(\frac{1}{x},\nu\right),$$ where $\omega^{-1}$ is the inverse of $\omega$ with respect to the  variable $x$. Pointwise, $\alpha(x,\nu)\to \frac{1}{x}$, as $\nu\to 0$. For every $\delta>0$, we get $x^{1+\delta}\alpha(x,\nu)\to x^\delta$, as $\nu\to 0$, \emph{uniformly in $x>0$}. The asymptotic behavior, as $x\to 0$, is qualitatively different in the case $\nu=0$ and $\nu\neq 0$:
\begin{equation}\label{eq:i}
\alpha(x,\nu)=\begin{cases}\frac{1}{x},&\nu=0,\\
\frac{1}{\nu}(-\log x)+\frac{\log\nu}{\nu}+\mathbb R_{\nu}[[x]],&\nu\neq 0.
\end{cases}
\end{equation}

Here and in the sequel $\R[[x]]$ denotes a formal expansion in powers of $x$ and $\R\{x\}$ an analytic germ in $x$. The notation $\mathbb R_{\nu}[[x]]$ resp. $\mathbb R_{\nu}\{x\}$ stands for a formal resp. analytic series in $x$ with coefficients analytic germs in $\nu$, that is, for $\mathbb R\{\nu\}[[x]]$ resp. $\mathbb R\{\nu\}\{x\}$.

\begin{defi}[The square root compensator] \label{def:pet}
For $\nu$ small by absolute value and $x>0$, we define
$$
\tilde\eta(x,\nu):=\sqrt{x+\nu}-\sqrt\nu,
$$
and call $\tilde\eta$ the \emph{square root-type compensator.}
\end{defi}
The asymptotic expansion of $\tilde \eta$, as $x\to 0$, changes qualitatively as $\nu$ changes from zero:
\begin{equation}\label{eq:dod}
\tilde\eta(x,\nu)=\begin{cases} \sqrt x,&\ \nu=0,\\
\frac{x}{\sqrt\nu}+\sqrt\nu\,\frac{x^2}{\nu^2}\,\mathbb R\{\frac{x}{\nu}\},&\ \nu> 0,\ x\to 0.
\end{cases}
\end{equation}
Note that $\tilde\eta$ is \emph{small}, for small $x$. Moreover, it can easily be checked that $\tilde\eta(x,\nu)\to\sqrt{x}$ uniformly in $x$, as $\nu\to 0+$.
\bigskip

Let
$$
a(\nu):=\frac{1-e^{-\nu}}{\nu },\ |\nu|<\delta.
$$
Note that $a\in\mathbb R\{\nu\}$ and $a(0)=1$. Therefore $a$ is bounded.
\medskip
\subsection{Model family}\label{subsec:mod}
Consider the model family \eqref{eq:model}, for $\nu\in[0,\delta)$. Let $f_\nu^{mod}$ be its time-one map and $g_\nu^{mod}:=\mathrm{id}-f_{\nu}^{mod}$ its displacement function. 

\medskip
Let $\ell^c(T_{\varepsilon,\nu})$, $\nu\in[0,\delta),$ be the continuous lengths of the tails for the orbits of time-one maps $f_\nu^{mod}$ for the unfolding \eqref{eq:model}, with initial condition $x_0>0$, as defined in Subsection~\ref{sec:conttime} and let $\theta_c(x)=x+c$ denote the translation by $c\in\mathbb R$.

\begin{prop}[Chebyshev system for the model family]\label{prop:pm} 
Let
\begin{equation}\label{eq:eta}
\eta(2\varepsilon,\nu):=\theta_{-\sqrt\nu}\circ (g_\nu^{mod})^{-1}(2\varepsilon),\ { \varepsilon> 0.}
\end{equation}
In the compensator variable $\eta\geq 0$, the continuous length $\eta\mapsto \ell^c(T_{\varepsilon,\nu})$, admits an asymptotic expansion,  uniform in the 
parameter $\nu\in[0,\delta)$, as $\eta\to 0$, in the following system:
$$
\left\{I(\nu,\eta)\eta,\ I(\nu,\eta)\eta^2,\ I(\nu,\eta)\eta^3,\ldots\right\},$$
 which becomes Chebyshev after division by the first term $I(\nu,\eta)\eta$. Here, 
\begin{equation}\label{eq:iii}
I(\nu,\eta):=\alpha(\eta,2\sqrt\nu)+\frac{\rho(\nu)}{2}\log\left(\eta^2+2\sqrt\nu\cdot\eta \right)-\Psi_\nu^{mod}(x_0),
\end{equation}
where the determination of $\Psi_\nu^{mod}$ from \eqref{eq:fat} is used.

Furthermore, there exist $\delta,\,d>0$ such that $I(\nu,\eta)$ is bounded from zero in a neighborhood $\eta\in (0,d)$, uniformly in $\nu\in[0,\delta)$. More precisely, for $\eta\to 0+$ the expansion is given by: 
\begin{align}\label{eq:okonule}
\ell^c(T_{\varepsilon,\nu})&= I(\nu,\eta)g_\nu^{mod}(\eta+\sqrt\nu)=\nonumber\\
&=\left(1-e^{-\frac{2\sqrt\nu}{1-\rho(\nu)\sqrt\nu}}\right)\cdot I(\nu,\eta)\eta+\\
&+e^{-\frac{2\sqrt\nu}{1-\rho(\nu)\sqrt\nu}}\, a\left(\frac{2\sqrt\nu}{1-\rho(\nu)\sqrt\nu}\right) \frac{1+\rho(\nu)\sqrt\nu}{(1-\rho(\nu)\sqrt\nu)^2}\cdot I(\nu,\eta)\eta^2+o_\nu(I(\nu,\eta)\eta^2).\nonumber
\end{align}
Here, $\rho(\nu)$ is the residual invariant from the normal form \eqref{eq:model}, and $o_\nu(I(\nu,\eta)\eta^2)$ means that $\lim_{\eta\to 0}\frac{o_\nu(I(\nu,\eta)\eta^2)}{I(\nu,\eta)\eta^2}=0$ uniformly in $\nu\in[0,\delta).$

\end{prop}


\begin{obs} Note that a similar expansion is obtained if we choose the initial point $x_0<0$, and the inverse orbit converging to the other  (repelling) fixed point $-\sqrt\nu$. Then $g_\nu^{mod}=\mathrm{id}-(f_{\nu}^{mod})^{-1}$.
\end{obs}

\begin{obs} Note that, by \eqref{eq:eta}, $\eta\geq 0$ for values $\varepsilon\geq 0$. At the bifurcation value $\nu=0$, the first non-zero coefficient of expansion \eqref{eq:okonule} is the second coefficient, 
signaling the multiplicity  $2$ of this bifurcation. However, analysing the signs of the first two coefficients of \eqref{eq:okonule} (both positive) we see that the second zero point is negative and is not realized in the unfolding for $\eta\inf0,\delta)$. This is not surprising since, for all parameter values $\nu$, the tube function is strictly positive for $\varepsilon>0$ and its only zero point is $\varepsilon=0$. 

If we instead choose the bilateral interval $\eta\in (-\delta,\delta)$ for the expansion, corresponding to positive and negative values of $\varepsilon$, we again get only one  true zero point $\eta=0$, i.e. $\varepsilon=0$ of the tube function. Although $\frac{\ell^c(T_{\varepsilon,\nu})}{I(\nu,\eta)}$ expands in the true multiplicity 2 Chebyshev scale, $\{\eta,\eta^2,\ldots\}$, thus surely realizing two zero points in the generic unfolding in the bilateral interval $\nu\in(-\delta,\delta)$, the function $I(\nu,\eta)$ explodes at these zero points, so that the product does not have a zero point except $\nu=0$. Indeed, function $I(\nu,\eta)$ comes from the Fatou coordinate, that explodes exactly at singularities of the vector field, and exactly these singularities are also the zeros of the displacement function $g_\nu(\eta+\sqrt\nu)$, see \eqref{eq:okonule}.

The similar analysis can be done for $k$-unfoldings, where we get multiplicity of the expansion of the tube function equal to $k$, although only the point $\eta=0$ i.e. $\varepsilon=0$ is the true zero point of the tube function for all values of the parameter. It is again due to the term $I_{\nu,\eta}$ that vanishes at all singular points of the field (=zero points of the displacement function) in the unfolding.
\end{obs}

\medskip


The variable $\eta$ is a \emph{small} variable (in the sense that it is tends to $0$ as $(\varepsilon,\nu)\to 0$) and behaves asymptotically as the square root compensator $\tilde\eta$ from Definition~\ref{def:pet}. Indeed, by Lemma~\ref{prop:ik},  it follows that
\begin{equation}\label{eq:etatilde}
\lim_{(\varepsilon,\nu)\to (0,0)}\frac{\eta(2\varepsilon,\nu)}{\tilde\eta(2\varepsilon,\nu)}=\lim_{(\varepsilon,\nu)\to (0,0)}\frac{\eta(2\varepsilon,\nu)}{\tilde\eta(2\varepsilon,c^2(\nu))}\frac{\tilde\eta(2\varepsilon,c^2(\nu))}{\tilde\eta(2\varepsilon,\nu)}=1,
\end{equation}
since $\lim_{\nu\to 0}\frac{\nu}{c^2(\nu)}=1$.
\medskip
The proof of Proposition~\ref{prop:pm} is given at the end of the subsection. For the proof, we need  Lemmas~\ref{prop:Fatou} and \ref{prop:displ}.
\begin{lem}[The time coordinate]\label{prop:Fatou} The time coordinate for family \eqref{eq:model} is $($up to an additive constant$)$ equal to:
\begin{align}
\Psi_\nu^{mod}(x)=&\alpha(x-\sqrt\nu,2\sqrt\nu)+\frac{\rho(\nu)}{2}\log(x^2-\nu)\nonumber\\
=&\left(\alpha(x,2\sqrt\nu)+\frac{\rho(\nu)}{2}\cdot\log(2\sqrt\nu \cdot x+x^2)\right)\circ\theta_{-\sqrt\nu}(x).\label{eq:haj}
\end{align}
Moreover,
\begin{equation}\label{eq:ij}
\Psi_\nu^{mod}(x)\sim_{x\to \sqrt\nu}\begin{cases}\frac{1}{x},&\nu=0,\\
\left(\frac{\rho(\nu)}{2}-\frac{1}{2\sqrt\nu}\right)\log(x-\sqrt\nu),&\nu\neq 0.
\end{cases}
\end{equation}
\end{lem}
Here and in the sequel the relation $\sim$ between two functions means that their quotient tends to $1$ at the limit, which is considered. More generally, it is also used to denote an asymptotic expansion.
\begin{proof}The time coordinate $\Psi_\nu^{mod}$ is computed as antiderivative in variable $x$ (determined up to an additive constant) of $\frac{1}{F_{mod}(\nu,x)}$. We get:
\begin{align}\label{eq:fat}
\Psi_\nu^{mod}(x)&=
\frac{1}{2\sqrt\nu}\ln\frac{x+\sqrt\nu}{x-\sqrt \nu}+\frac{\rho(\nu)}{2}\log(x^2-\nu)\nonumber\\
&=\frac{1}{2\sqrt\nu}\ln\left(1+\frac{2\sqrt\nu}{x-\sqrt\nu}\right)+\frac{\rho(\nu)}{2}\log(x^2-\nu),
\end{align}
and substitute $\alpha$ from Definition~\ref{def:lc}.
In case $\nu\neq 0$, and $x\to \sqrt{\nu}$, we have:
\begin{equation}\label{eq:ii}
\log(x^2-\nu)=\!\!\left(\log(x-\sqrt\nu)+\log(2\sqrt\nu)+\log\big(1+\frac{x-\sqrt\nu}{2\sqrt\nu}\big)\right)=\log(x-\sqrt\nu)+O(1).
\end{equation}
Therefore, by \eqref{eq:fat} and \eqref{eq:ii}, we deduce \eqref{eq:ij}.
\end{proof}
\begin{obs} Note that the above formula \eqref{eq:fat} is valid, for all values of $\nu\in(-\delta,\delta)$. In case $\nu<0$, $\sqrt\nu=\sqrt{|\nu|}i$ is  pure imaginary. The formula can be rewritten as:
$$
\Psi_\nu^{mod}(x)=\begin{cases}\frac{1}{x}+\rho(0)\cdot\log x,& \nu=0,\\
\frac{1}{2\sqrt\nu}\ln\left(1+\frac{2\sqrt\nu}{x-\sqrt\nu}\right)+\frac{\rho(\nu)}{2}\log(x^2-\nu),& \nu>0,\\
-\frac{1}{\sqrt{|\nu|}}\arctan\frac{x}{\sqrt{|\nu|}}+\frac{\rho(\nu)}{2}\log(x^2-\nu),& \nu<0.
\end{cases}
$$
\end{obs}
\medskip

\begin{lem}[The time-one map and the displacement germ]\label{prop:displ}For family \eqref{eq:model}, the following Taylor expansions at the fixed point\footnote{A similar expansion can be obtained near the other, symmetric fixed point $x=-\sqrt \nu$.} $x=\sqrt\nu$ hold:
\begin{enumerate}
\item For the time-one map $f_\nu^{mod}\in\mathrm{Diff}(\mathbb R,0)$:
\begin{align}\label{eq:tom}
&f_{\nu}^{mod}(x)\sim\sqrt\nu+e^{-\frac{2\sqrt\nu}{1-\rho(\nu)\sqrt\nu}}\cdot (x-\sqrt\nu)-\nonumber\\
&\qquad-e^{-\frac{2\sqrt\nu}{1-\rho(\nu)\sqrt\nu}}\cdot a\big(\frac{2\sqrt\nu}{1-\rho(\nu)\sqrt\nu}\big)\cdot\frac{1+\rho(\nu)\sqrt\nu}{(1-\rho(\nu)\sqrt\nu)^2}\cdot(x-\sqrt\nu)^2+\nonumber\\
&\qquad\qquad\qquad\qquad +(x-\sqrt\nu)^3\cdot \mathbb R_{\nu}\{(x-\sqrt\nu)\},\qquad \ \nu\in[0,\delta),\ x\to\sqrt\nu,
\end{align}
\item
For the displacement function $g_\nu^{mod}:=\mathrm{id}-f_\nu^{mod}\in\mathrm{Diff}(\mathbb R,0)$:
\begin{align}\label{eq:hi}
g_{\nu}^{mod}(x)\sim&\left(1-e^{-\frac{2\sqrt\nu}{1-\rho(\nu)\sqrt\nu}}\right)\cdot (x-\sqrt\nu)+\nonumber\\
&\quad\ +e^{-\frac{2\sqrt\nu}{1-\rho(\nu)\sqrt\nu}}\cdot a\left(\frac{2\sqrt\nu}{1-\rho(\nu)\sqrt\nu}\right)\cdot\frac{1+\rho(\nu)\sqrt\nu}{(1-\rho(\nu)\sqrt\nu)^2}\cdot(x-\sqrt\nu)^2+\nonumber\\
&\quad\qquad\qquad+(x-\sqrt\nu)^3\cdot \mathbb R_{\nu}\{(x-\sqrt\nu)\},\quad \nu\in[0,\delta),\ x\to\sqrt\nu.
\end{align}
\end{enumerate}
The Taylor coefficients of $f_\nu^{mod}$ and $g_\nu^{mod}$  belong to $\mathbb R\{\sqrt\nu\}$ $($are analytic at $0$ in $\sqrt\nu)$.
\end{lem}
Note that, for $\nu=0$, $f_\nu^{mod}$ has a parabolic double fixed point $x=0$, and, for $\nu> 0$, two single hyperbolic fixed points at $\pm\sqrt{\nu}$.

\begin{proof}
The time-one maps $f_\nu^{mod}=\mathrm{Exp}(F^{mod}(x,\nu)\frac{d}{dx}).\mathrm{id}$ of system \eqref{eq:model} are parabolic ($\nu=0$) or hyperbolic ($\nu>0$) analytic germs at $x=0$. Moreover, $f_\nu^{mod}\to f_0^{mod}$ uniformly on some interval $x\in(-d,d)$, $d>0$, as $\nu\to 0+$. Indeed, by simple computation it can be verified that the singularity of $f_\nu^{mod}$ converges to $1$, as $\nu\to 0$. It can be checked by the operator exponential formula above that the coefficients $a_k(\nu)$ of monomials $x^k$, $k\geq 0$, in the Taylor expansion of $f_\nu^{mod}(x)=\sum_{k=0}^{\infty}a_k(\nu)x^k$ converge towards coefficients $a_k(0)$ of $f_0^{mod}(x)=\sum_{k=0}^{\infty}a_k(0)x^k$, as $\nu\to 0$. Therefore, to get \eqref{eq:tom}, it suffices to get the Taylor expansion of $f_\nu^{mod}$ at $\sqrt\nu$ in the case when $\nu> 0$. 

Suppose therefore $\nu> 0$. The time-one map $f_\nu$ is obtained from the time coordinate by the Abel equation $f_\nu=(\Psi_\nu^{mod})^{-1}(\Psi_\nu^{mod}+1)$. First we compute (the first few terms) of the inverse $(\Psi_\nu^{mod})^{-1}$. 

When $\nu\neq 0$ and $x\to\sqrt\nu$, by \eqref{eq:fat} $\Psi_\nu^{mod}$ admits the following expansion, as $x\to\sqrt\nu$:
\begin{align*}
\Psi_\nu^{mod}&(x)=\left(\frac{\rho(\nu)}{2}-\frac{1}{2\sqrt\nu}\right)\log(x-\sqrt\nu)+\left(\frac{\rho(\nu)}{2}+\frac{1}{2\sqrt\nu}\right)\log(2\sqrt\nu)+\\
&\qquad\qquad \qquad\qquad\qquad\qquad\quad\qquad+\left(\frac{\rho(\nu)}{2}+\frac{1}{2\sqrt\nu}\right)\log(1+\frac{x-\sqrt\nu}{2\sqrt\nu})=\\
\sim&\left(\frac{\rho(\nu)}{2}-\frac{1}{2\sqrt\nu}\right)\log(x-\sqrt\nu)+\left(\frac{\rho(\nu)}{2}+\frac{1}{2\sqrt\nu}\right)\log(2\sqrt\nu)+\mathbb R_\nu[[(x-\sqrt\nu)]].
\end{align*}
Denote the above coefficients by: $K_\pm(\nu):=\frac{\rho(\nu)}{2}\pm\frac{1}{2\sqrt\nu}$, $K(\nu):=K_+(\nu)\log(2\sqrt\nu)$. Then, for the expansion of the inverse (where $\frac{y}{K_-(\nu)}\to -\infty$), we get:
\begin{align*}
(\Psi_\nu^{mod})^{-1}(y)&\sim\sqrt\nu+e^{\frac{y-K(\nu)}{K_-(\nu)}}-\frac{K_+(\nu)}{K_-(\nu)\cdot 2\sqrt\nu}e^{2\frac{y-K(\nu)}{K_-(\nu)}}+e^{3\frac{y-K(\nu)}{K_-(\nu)}}\mathbb R_{\nu}\big[\big[e^{\frac{y-K(\nu)}{K_-(\nu)}}\big]\big].
\end{align*}
Therefore,
\begin{align*}
&f_{\nu}^{mod}(x)=(\Psi_{\nu}^{mod})^{-1}(1+\Psi_\nu^{mod}(x))\\
&\sim\sqrt\nu+e^{-\frac{2\sqrt\nu}{1-\rho(\nu)\sqrt\nu}}\cdot (x-\sqrt\nu)-\\
&\qquad\qquad-e^{-\frac{2\sqrt\nu}{1-\rho(\nu)\sqrt\nu}}\cdot a\left(\frac{2\sqrt\nu}{1-\rho(\nu)\sqrt\nu}\right)\cdot\frac{1+\rho(\nu)\sqrt\nu}{(1-\rho(\nu)\sqrt\nu)^2}\cdot(x-\sqrt\nu)^2+\\
&\qquad\quad\quad\quad\qquad\qquad\qquad\qquad\qquad+(x-\sqrt\nu)^3\cdot \mathbb R_{\nu}\{(x-\sqrt\nu)\},\ x\to\sqrt\nu.
\end{align*}
Note that $\rho(\nu)$ is a bounded function on $\nu\in[0,\delta)$, so $\rho(\nu)\sqrt\nu\to 0$, as $\nu\to 0+$.

\medskip

\end{proof}



\smallskip



\begin{proof}[Proof of Proposition~\ref{prop:pm}] We use the formula for the continuous tail \eqref{eq:druga}, and change the variable from $\varepsilon$ to $\eta$ given by \eqref{eq:eta}, i.e. verifying $2\varepsilon=g^{mod}_\nu(\eta+\sqrt{\nu})$:

$$
\ell^c(T_{\varepsilon,\nu})=\left(\Psi_\nu^{mod}((g_{\nu}^{mod})^{-1}(2\varepsilon))-\Psi_\nu^{mod}(x_0)\right)\cdot 2\varepsilon.
$$
We denote $I(\eta,\nu):=\Psi_\nu^{mod}\left((g_{\nu}^{mod})^{-1}(2\varepsilon)\right)-\Psi_\nu^{mod}(x_0)$. Substituting $\eta=\theta_{-\sqrt\nu}\circ (g_{\nu}^{mod})^{-1}(2\varepsilon)$ in $\Psi_\nu^{mod}$ given in \eqref{eq:haj} in Lemma~\ref{prop:Fatou}, we get $I(\eta,\nu)$, as in \eqref{eq:iii}. On the other hand, 
from $2\varepsilon=
g_\nu^{mod}(\eta+\sqrt\nu)$,  expansion \eqref{eq:okonule} follows from \eqref{eq:hi} of Lemma~\ref{prop:displ}.

Note that there exist $d,\,\delta>0$ such that $\eta\mapsto \Psi_\nu^{mod}(\eta+\sqrt\nu)$ is injective on $\eta\in[0,d)$, $d>0$ small, for all $\nu\in[0,\delta)$. Indeed, $(\Psi_\nu^{mod})'(\eta+\sqrt\nu)=\frac{\rho(\nu)(y+\sqrt\nu)-1}{y^2+2\sqrt \nu y}$ is strictly negative, for $\nu\geq 0,\,\eta>0$ sufficiently small.  Choosing $x_0$ inside this interval $\sqrt{\nu}+[0,d)$, $\Psi_{\nu}^{mod}(\theta_{\sqrt\nu}\circ \theta_{-\sqrt\nu}\circ (g_\nu^{mod})^{-1}(2\varepsilon))=\Psi_\nu^{mod}(\theta_{\sqrt\nu}\circ\theta_{-\sqrt\nu}(x_0))$, if and only if 
$$
\eta=\theta_{-\sqrt\nu}\circ (g_\nu^{mod})^{-1}(2\varepsilon)=\theta_{-\sqrt\nu} (x_0).
$$
Therefore, for $\eta$ sufficiently small and $\nu\in [0,\delta)$, the function $I(\nu,\eta)$ has no zeros.
The scale is Chebyshev, since, apart from the common nonzero factor $I(\nu,\eta)$, it is the classical monomial scale. 
The result now follows.

Finally, $\lim_{\eta\to 0}\frac{o_\nu(I(\nu,\eta)\eta^2)}{I(\nu,\eta)\eta^2}=0$ uniformly in $\nu\in[0,\delta)$, since the family of displacement functions $g_\nu^{mod})$ depends analytically on $\nu\in[0,\delta)$.
\end{proof}

\subsection{Generic saddle-node families}\label{subsec:nonmod}
Let \eqref{eq:polje}, $\nu\in[0,\delta)$, be a generic saddle-node analytic\footnote{In fact, the family of germs is assumed analytic in $\nu$ for all $\nu\in(-\delta,\delta)$, but we restrict only to the semi-closed subinterval for which two hyperbolic points unfold from the parabolic point at $\nu=0$.} germ of a family, satisfying generic conditions \eqref{eq:gene}. Let $f_\nu$ denote its time-one maps, and let $f_\nu^\mathrm{mod}$ denote the time-one maps of its analytic model \eqref{eq:model}. Let $\Phi(\nu,x):=(\varphi_\nu(x),h(\nu))$ be the analytic change of variables conjugating \eqref{eq:polje} to its model \eqref{eq:model}, where $h(0)=0$ and $h$ is an analytic diffeomorphism at $0$, and $\varphi_\nu(x)=a_0(\nu)+a_1(\nu)x+o_\nu(x^2)$ an analytic  diffeomorphism at $x=0$, with $a_1(\nu)\neq 0$, $\nu\in[0,\delta)$.  Here, by \cite{MRouss}, the change of variables in analytic for $\nu\in(0,\delta)$ and continuous at $\nu=0$ (see the notion of \emph{weak conjugacy} in \cite{MRouss}) and therefore \emph{uniform in $\nu\in[0,\delta)$} , i.e. $o_\nu(x^2)$ means that $\lim_{x\to 0}\frac{o_\nu(x^2)}{x^2}=0$ \emph{uniformly} in $\nu\in[0,\delta)$. Then, by \eqref{eq:tmaps}, the initial time-one map $f_\nu$ verifies:
\begin{equation*}
f_{\nu}=\varphi_\nu^{-1}\circ f_{h(\nu)}^{\mathrm{mod}}\circ\varphi_\nu.
\end{equation*}
 Let $x_{1,2}^{\nu}$ denote the fixed points of $f_{\nu}$. It holds that $x_{1,2}^\nu\to 0,\ \nu \to 0$. Without loss of generality, we assume that the fixed point $x_1^{\nu}$ is positive and attractive. For the germ at $0$ of the time coordinate, it holds:
\begin{equation}\label{eq:phi}
\Psi_{\nu}=\Psi_{h(\nu)}^{\mathrm{mod}}\circ\varphi_\nu.
\end{equation}
As in the introduction, without loss of generality, we assume that $F_{\nu}(0,0)>0$ so that $h(\nu)>0$, for $\nu>0$. Note that $\varphi_\nu(x_1^{\nu})=\sqrt{h(\nu)}$, or $-\sqrt{h(\nu)}$ and we suppose that $\varphi_\nu(x_1^{\nu})=\sqrt{h(\nu)}$. 
\bigskip

Let $\ell^c(T_{\varepsilon,\nu})$, $\nu\in[0,\delta)$, be the continuous length of the tail of the $\varepsilon$-neighborhood of the orbit $\mathcal O_{f_\nu}(x_0)$, for $x_0>0$ sufficiently small. Then, there exists $\delta>0$ such that $x_0$ is attracted by $x_1^\nu$, for all sufficiently small $\nu\in[0,\delta)$. Analogously, we could have considered initial point $x_0<0$ repelled from $x_2^{\nu}$ (i.e. attracted to it by $f_\nu^{-1}$), $\nu\in[0,\delta)$, and the orbit of the inverse $\mathcal O_{f_\nu^{-1}}(x_0)$.

Let $\Phi=(h,\varphi_\nu)$, $\nu\in[0,\delta)$, be the normalizing germ of change of variables reducing a given saddle-node field \eqref{eq:polje} to its model \eqref{eq:model}, and let $C(\nu):=\varphi_\nu'(x_1^\nu)\neq 0$ be the coefficient of the linear term of $\varphi_\nu$ at the point $x_1^\nu$.
Let 
\begin{equation}\label{eq:k}
k_\nu:=\theta_{-\sqrt{h(\nu)}}\circ \varphi_{\nu}\circ \theta_{x_1^\nu}.
\end{equation}
Then
$
k_\nu(x)=C(\nu)x+o_{\nu}(x),\ x\in (-d,d), \ x\to 0, \ \nu\in[0,\delta),
$
is analytic germ (i.e. germ of an analytic family, analytic in $x\in(-d,d)$ and in $\nu\in [0,\delta)$).

\begin{thmBB}[Chebyshev system for generic cases]\label{thm:A} Let
\begin{equation}\label{eq:eta1}
\eta(2\varepsilon,\nu):=\theta_{-x_1^{\nu}} \circ g_{\nu}^{-1}(2\varepsilon),\ \varepsilon\sim 0,
\end{equation}
where $g_\nu:=\mathrm{id}-f_\nu$. In the variable $\eta\geq0$, the continuous length $\eta\mapsto \ell^c(T_{\varepsilon,\nu})$ admits a uniform asymptotic expansion in the following  system:
\begin{equation}\label{eq:Cheb}
\left\{I(h(\nu),k_{\nu}(\eta))\eta,\ I(h(\nu),k_\nu(\eta))\eta^2,\ I(h(\nu),k_{\nu}(\eta))\eta^3,\ldots\right\},
\end{equation}as $\eta\to 0$, 
where $I(\nu,\eta)$ is as given in \eqref{eq:iii}, which becomes Chebyshev after division by the common term $I(h(\nu),k_\nu(\eta))\eta$. There exist $\delta,\ d>0$ such that the common term $I(h(\nu),k_\nu(\eta))$ is away from zero, for $\eta\in [0,d)$, uniformly in $\nu\in[0,\delta)$.

More precisely, the expansion is given by:
\begin{align}\label{eq:expa1}
&\ell^c(T_{\varepsilon,\nu})= I(h(\nu),k_\nu(\eta))\cdot g_\nu(\eta+x_1^\nu)=\nonumber\\
&=\left(1-e^{-\frac{2\sqrt{h(\nu)}}{1-\rho(h(\nu))\sqrt{h(\nu)}}}\right)\cdot I(h(\nu),k_\nu(\eta))\eta+\\
&\qquad\qquad\qquad+ c_2(\nu)\cdot I(h(\nu),k_\nu(\eta))\eta^2+o_\nu(I(h(\nu),k_\nu(\eta))\eta^2),\ \eta\to 0.\nonumber
\end{align}
Here, $c_2(0)\neq 0$, and notation $o_\nu(.)$ means that the limit is uniform in $\nu\in[0,\delta).$

\end{thmBB}

The following lemma is used in the proof of Theorem~B.
\begin{lem}\label{lem:gee} The coefficient of the linear term in the expansion of $g_\nu$ around $x_1^{\nu}$ is the same as the coefficient of the linear term in the expansion of $g_{h(\nu)}^{\mathrm{mod}}$ around $\sqrt{h(\nu)}$. There are no free coefficients. Moreover, the coefficient of the quadratic term in the expansion of $g_{\nu}$ around $x_1^{\nu}$ at the bifurcation value $\nu=0$ is nonzero. 
\end{lem}

\begin{proof}
Since $\varphi_0'(0)\neq 0$, due to the continuity of $(x,\nu)\mapsto \varphi_\nu'(x)$, it follows that $\varphi_\nu'(x_{1}^{\nu})\neq 0$, for $\nu$ sufficiently small. 
Therefore, the following expansion holds:
\begin{equation}\label{eq:l}
\varphi_{\nu}(x)=\sqrt{h(\nu)}+C(\nu)(x-x_1^{\nu})+o_\nu(x-x_1^{\nu}),\ C(\nu)\neq 0.
\end{equation}
It follows by \eqref{eq:tmaps} and \eqref{eq:l} that
\begin{align}\label{eq:j}
&f_{\nu}'(x_1^{\nu})=(\varphi_\nu^{-1})'(f_{h(\nu)}^{\mathrm{mod}}\circ\varphi_\nu)(x_1^{\nu})\cdot (f_{h(\nu)}^{\mathrm{mod}})'(\varphi_\nu(x_1^{\nu}))\cdot\varphi_\nu'(x_1^{\nu})=\nonumber\\
&=\frac{1}{\varphi_\nu'(x_1^{\nu})}\cdot (f_{h(\nu)}^{\mathrm{mod}})'(\sqrt{h(\nu)})\cdot \varphi_\nu'(x_1^{\nu})=(f_{h(\nu)}^{\mathrm{mod}})'(\sqrt{h(\nu)}),\nonumber\\
&f_0''(0)=\frac{\varphi_0''(0)}{C(0)}\left((f_0^{\mathrm{mod}})'(0)-(f_0^{\mathrm{mod}})'(0)^2\right)+C(0)(f_0^\mathrm{mod})''(0)=\nonumber\\
&\ \qquad\qquad=C(0)(f_0^\mathrm{mod})''(0)=2C(0)\neq 0.
\end{align}
The last line follows since $(f_0^\mathrm{mod})'(0)=1$ (tangent to the identity). 

The fact that the expansion, i.e. the notation $o_\nu$ is uniform in $\nu\in[0,\delta)$ follows from the fact that the expansion of $f_\nu^{mod}(x)$ is uniform in $\nu\in[0,\delta)$ since it is an analytic family, and that the change of variables $\varphi_\nu$ is analytic in $\nu\in(0,\delta)$ and continuous at $\nu=0$, and therefore also expands uniformly in $\nu\in[0,\delta)$. As a consequence, the family $g_\nu(\eta)$ expands uniformly in $\nu\in[0,\delta)$ as $\eta\to 0$. \end{proof}

\begin{proof}[Proof of Theorem~B] We have:
\begin{equation}\label{eq:insert}
\ell^c(T_{\varepsilon,\nu})=(\Psi_\nu(g_{\nu}^{-1}(2\varepsilon))-\Psi_\nu(x_0))\cdot 2\varepsilon.
\end{equation}
 Put $\eta:=\theta_{-x_1^\nu} \circ g_{\nu}^{-1}(2\varepsilon) $, as in \eqref{eq:eta1}. Therefore, $g_{\nu}^{-1}(2\varepsilon)=\theta_{x_1^\nu}\circ \eta$. By \eqref{eq:phi}, we get
\begin{equation}
\Psi_\nu(g_{\nu}^{-1}(2\varepsilon))=\Psi_{h(\nu)}^{\mathrm{mod}}\circ\varphi_{\nu}(\theta_{x_1^\nu}\circ \eta).
\end{equation}
Let $k_\nu=C(\nu)y+o_\nu(y)$ be as defined in \eqref{eq:k}. We then have,  by \eqref{eq:iii} and \eqref{eq:haj}:
\begin{equation}\label{eq:help1}
\Psi_\nu(g_{\nu}^{-1}(2\varepsilon))=\left(\Psi_{h(\nu)}^{\mathrm{mod}}\circ\theta_{\sqrt{h(\nu)}}\right)(k_\nu(\eta))=I(h(\nu),k_\nu(\eta))+\Psi_{h(\nu)}^{mod}(x_0).
\end{equation} 

On the other hand, $2\varepsilon=g_\nu(\eta+x_1^\nu)$. Using Lemma~\ref{lem:gee} to get the first terms of the expansion of $g_\nu$ and inserting it together with \eqref{eq:help1} in \eqref{eq:insert}, we get expansion \eqref{eq:expa1}.

Finally, since $h(0)=0$, $k_\nu(\eta)=O(\eta)$, $h$ and $k_\nu$ are diffeomorphisms of some positive open neighborhoods of $0$, $k_\nu$ depends continuously on $\nu$ and $I(\nu,\eta)$ is non-zero, for $\eta\in [0,d)$ and $\nu\in[0,\delta)$, the same holds, for $(\nu,\eta)\mapsto I(h(\nu),k_\nu(\eta))$, possibly in smaller neighborhoods.
\end{proof}

Note that it is more convenient in applications to consider the standard length $\ell(T_{\varepsilon,\nu})$ instead of continuous length $\ell^c(T_{\varepsilon,\nu})$. Let $G:\mathbb[0,+\infty)\to[0,+\infty)$ be the periodic function of period $1$, on $[0,1)$ given by $G(s)=1-s,\ s\in(0,1)$ and $G(0)=0$. We have the following corollary:
\begin{cory}[Expansion of the standard length $\ell(T_{\varepsilon,\nu})$] \label{cor:discrete} Under assumptions of Theorem~B,
the length $\eta\mapsto \ell(T_{\varepsilon,\nu})$ admits a uniform asymptotic expansion in the Chebyshev system \eqref{eq:Cheb}, but with $I(h(\nu),k_{\nu}(\eta))$ replaced by 
$$(\mathrm{id}+G)\left(I(h(\nu),k_{\nu}(\eta))\right),$$ which also does not have zero points for any $\nu$ on some uniform interval $\eta\in[0,d)$, $d>0$. The asymptotic expansion is given by \eqref{eq:expa1}, up to the same modification.
\end{cory}

\begin{proof} This follows directly using the relation between continuous $\tau_{\varepsilon}^\nu$ and discrete critical time $n_{\varepsilon}^{\nu}$, which is its integer part:
$$
n_{\varepsilon}^{\nu}-\tau_{\varepsilon}^{\nu}=G(\tau_{\varepsilon}^\nu),\ \nu\in[0, \delta).
$$
For details, see Subsection~\ref{sec:conttime} and \cite{frenki}. Therefore, $\mathcal\ell(T_{\varepsilon,\nu})-\mathcal\ell^c(T_{\varepsilon,\nu})=G(\tau_{\varepsilon}^{\nu})\cdot 2\varepsilon$, and the claim follows from \eqref{eq:Cheb} and \eqref{eq:expa1} in Theorem~B. 

Moreover, since $I(h(\nu),k_{\nu}(\eta))$ is big (uniformly in $\nu$), for small $\eta\in[0,d)$, and $G$ is bounded inside $[0,1]$, it follows that $(id+G)\Big(I(h(\nu),k_{\nu}(\eta))\Big)$ is  nonzero in some small neighborhood $\eta\in[0,d)$, uniform in $\nu$.
\end{proof}

\begin{proof}[Proof of Theorem A] After the change of variables $\eta:=\theta_{-x_1^\nu} \circ g_{\nu}^{-1}(2\varepsilon  )$, we have $$\ell^c(T_{\varepsilon,\nu})(\eta)=(\Psi_\nu(\eta+x_1^\nu)-\Psi_{\nu}(x_0))\cdot g_\nu(\eta+x_1^\nu).$$ Since the first factor in brackets, for any $\nu\in[0,d)$, does not have zero points on some uniform (in $\nu$) positive interval around $\eta=0$, the multiplicity (from the right) of the zero point $0$ of $\ell^c(T_{\varepsilon,0})$ in the variable $\eta$ in the unfolding is the same as the multiplicity (from the right) of the zero point $0$ of the displacement function $g_0(\eta)$ in the unfolding, which corresponds to the multiplicity (from the right) of the singular point $0$ of the saddle-node vector field in the unfolding \eqref{eq:polje}.

By Corollary~\ref{cor:discrete}, $\mathcal\ell(T_{\varepsilon,\nu})=\big(\mathrm{id}+G\big)(\Psi_\nu(\eta+x_1^\nu)-\Psi_{\nu}(x_0))\cdot g_\nu(\eta+x_1^\nu) $. As above, the first factor does not have zero points on some in $\nu$ uniform positive interval around $\eta=0$, and the same conclusion about multiplicities follows for $\mathcal\ell(T_{\varepsilon,\nu})$ instead of $\mathcal\ell^c(T_{\varepsilon,\nu})$. \end{proof}

\section{Precise forms of the expansions of the length function $\ell^c(T_{\varepsilon,\nu})$ for all parameter values}\label{sec:comp}
The expansion \eqref{eq:expa1} in Theorem B is valid for the whole bifurcation. The variable $\eta(2\varepsilon,\nu)$ is a \emph{compensator variable} that behaves asymptotically qualitatively differently, as $\varepsilon\to 0$, depending on the case $\nu=0$, or $\nu>0$, see \eqref{eq:dod} and \eqref{eq:etatilde}. 

By \eqref{eq:etatilde}, the variable $\eta$ given in \eqref{eq:eta1}, behaves essentially as the simpler compensator $\tilde\eta$ defined in Definition~\ref{def:pet}.

In Lemma~\ref{prop:ik}, we expand $\eta$ from \eqref{eq:eta1} in a simpler square root compensator variable $\tilde\eta$ and expand $\ell^c(T_{\varepsilon,\nu})$ in a Chebyshev system in this simpler compensator variable $\tilde\eta$ instead of $\eta$. In Theorem~C we then re-group the terms of this new expansion so that, for $\nu>0$, all terms in the same block merge to the same term of the asymptotic expansion in $\varepsilon$ at the bifurcation value $\nu=0$. Hence, we show that confluence of singularities leads to confluence of asymptotic terms in the expansion of $\varepsilon\mapsto \ell^c(T_{\varepsilon,\nu})$, as $\nu\to 0$.

In particular, from Theorem~C, in Subsection~\ref{ss:prvi}, we deduce the expansions of the continuous length $\ell^c(T_{\varepsilon,\nu})$ in $\varepsilon$, as $\varepsilon\to 0$, for each of the qualitatively different cases, $\nu>0$ and $\nu=0$. Theorem~C is then used in Remark~\ref{rem:zadnji} for reading the formal class of the unfolding from fractal data.


\medskip
In Theorem~C, we use another compensator variable:
\begin{equation}\label{eq:kapa}
\kappa(x,\nu):=\frac{1}{x+\nu}.
\end{equation}
It is related to the compensator $\alpha$ from Definition~\ref{def:lc} by the following formula:
$$
\frac{d}{dx}\alpha(x,\nu)=-\frac{1}{x}\kappa(x,\nu).
$$
Evidently, $\kappa(x,\nu)\to\frac{1}{x}$, as $\nu\to 0$, moreover, for every $\delta>0$, $x^\delta\kappa(x,\nu)\to x^{-1+\delta}$, uniformly, as $\nu\to 0$. 
\smallskip

Let $\nu\mapsto h(\nu)$, $\nu\mapsto C(\nu)$, $\nu\mapsto c_2(\nu)$, $\nu\in[0,\delta)$, be as in Theorem~B, and let
$$
r(\nu):=\frac{1-e^{-\frac{2\sqrt{\nu}}{1-\rho(\nu)\sqrt{\nu}}}}{2C(0)},\ \nu\in[0,\delta). 
$$
\smallskip
\begin{thmC} The expansion of the continuous length $\ell^c(T_{\varepsilon,\nu})$ in the variable $\tilde\eta:=\tilde\eta\left(\frac{2\varepsilon}{C(0)},r(h(\nu))\right)$, as $\tilde\eta\to 0$, can be written in the form:
\begin{align}\label{eq:ce}
&\ell^c(T_{\varepsilon,\nu})\sim \left(1-e^{-\frac{2\sqrt{h(\nu)}}{1-\rho(h(\nu))\sqrt{h(\nu)}}}\right)\cdot \Bigg\{\alpha(C(\nu)\tilde\eta,2\sqrt{h(\nu)})\cdot \tilde\eta+\Bigg.\nonumber\\
&\ \ \qquad\ +\frac{\rho(h(\nu))}{2}\sum_{k=0}^{\infty}a_k(\nu)\Big[\log\tilde\eta\cdot\tilde\eta^{k+1}+\log\big(\tilde\eta+\frac{2\sqrt{h(\nu)}}{C(\nu)}\big)\cdot\tilde\eta^{k+1}\Big]+\nonumber\\
&\ \ \qquad\ \Bigg.+\sum_{k=1}^\infty \Big[a_k(\nu)\cdot\alpha\big(C(\nu)\tilde\eta,2\sqrt{h(\nu)}\big)\cdot \tilde\eta^{k+1}+N_k^\nu\big(\tilde\eta,\kappa\big(C(\nu)\tilde\eta,2\sqrt{h(\nu)}\big)\big)\Big]\Bigg\}+\nonumber\\
&\ + c_2(\nu)\cdot \Bigg\{  [\alpha(C(\nu)\tilde\eta,2\sqrt{h(\nu)})\cdot \tilde\eta^2]+\Bigg.\nonumber\\
&\ \ \qquad\ +\frac{\rho(h(\nu))}{2}\sum_{k=0}^{\infty}b_k(\nu)\Big[\log\tilde\eta\cdot\tilde\eta^{k+2}+\log\big(\tilde\eta+\frac{2\sqrt{h(\nu)}}{C(\nu)}\big)\cdot\tilde\eta^{k+2}\Big]+\nonumber\\
&\ \ \qquad \ \Bigg.+\sum_{k=1}^\infty \Big[b_k(\nu)\cdot\alpha\big(C(\nu)\tilde\eta,2\sqrt{h(\nu)}\big)\cdot \tilde\eta^{k+2}+M_{k+1}^\nu\big(\tilde\eta,\kappa\big(C(\nu)\tilde\eta,2\sqrt{h(\nu)}\big)\big)\Big]\Bigg\}.
\end{align}
Here, $c_2(0)\neq 0$, $a_0(\nu)=1$ and $\beta_0(\nu)=1$, for all $\nu\in[0,\delta)$, and $N_k^{\nu},\,M_k^\nu$ are homogenous polynomials of degree $k$ whose coefficients depend on $\nu$. 
\end{thmC}

The expansion is written in  such a form that the terms that give the same power-logarithmic asymptotic term in $\tilde\eta$  for $\nu=0$, with their respective coefficients in $\nu$, are grouped together as a block inside square brackets. Note that, for a fixed $\nu>0$, each block is possibly \emph{infinite} in the sense that it can be further expanded asymptotically in a convergent power-logarithmic series in $\tilde\eta$, as $\tilde\eta\to 0$. For simplicity, in Theorem~C each block is written in a closed form, as a true function of $\tilde\eta$.

Moreover, by \eqref{eq:dod}, $\tilde\eta=\sqrt{\frac{2}{C(0)}}\varepsilon^{1/2}$ for $\nu=0$, and $\tilde\eta$ expands as an integer power series in $\varepsilon$, for  $\nu>0$, so complete expansions in the original variable $\varepsilon\to 0$, for $\nu=0$ and $\nu>0$ are given in Subsection~\ref{ss:prvi}.

\medskip
In the proof of Theorem~C we need the following lemmas:

\begin{lem}[The compensator variable $\eta$ expressed by $\tilde\eta$]\label{prop:ik} Let $\eta$ be as defined in \eqref{eq:eta1} for the field \eqref{eq:polje} and let $\tilde\eta$ be as in Definition~\ref{def:pet}. Then:
$$
\eta(2\varepsilon,\nu)=\chi_{\nu} \Big(\tilde\eta\big(\frac{2\varepsilon}{C(0)},r^2(h(\nu))\big)\Big),\ \nu\in[0,\delta),
$$
where $r(\nu):=\frac{1-e^{-\frac{2\sqrt{\nu}}{1-\rho(\nu)\sqrt{\nu}}}}{2C(0)},$ and\ $\chi_\nu$ is a germ of a real diffeomorphism tangent to the identity.
\end{lem}

\noindent Consequently, $\eta$ possesses a Taylor expansion in the variable $\tilde\eta\big(\frac{2\varepsilon}{C(0)},r^2(h(\nu))\big)$, and 
$$
\lim_{(\varepsilon,\nu)\to (0,0)}\frac{\eta(2\varepsilon,\nu)}{\tilde\eta(\frac{2\varepsilon}{C(0)},r^2(h(\nu)))}=1.$$
\begin{proof} 
From \eqref{eq:hi} and Lemma~\ref{lem:gee}, we write $g_{\nu}$ as:
\begin{align}
g_{\nu}(x)&=\left(\big(1-e^{-\frac{2\sqrt{h(\nu)}}{1-\rho(h(\nu))\sqrt{h(\nu)}}}\big)\cdot (x-x_1^\nu)+C(0)(x-x_1^\nu)^2\right)\circ \nonumber\\
&\quad\qquad\qquad\qquad\circ\big((x-x_1^{\nu})+\sum_{i=2}^{\infty}c_i(\nu)(x-x_1^{\nu})^i\big)=\nonumber\\
&=P_{\nu}\circ\psi_{\nu}\circ\theta_{-x_1^{\nu}}(x),\label{eq:mi}
\end{align}
where $P_{\nu}(x)=\left(1-e^{-\frac{2\sqrt{h(\nu)}}{1-\rho(h(\nu))\sqrt{h(\nu)}}}\right)\cdot x+C(0)x^2$, and  $\psi_{\nu}:=\mathrm{id}+\sum_{i=2}^{\infty}c_i(\nu)x^i$
is a germ of a real diffeomeorphism, tangent to the identity at $0$. Note that $1-e^{-\frac{2\sqrt{h(\nu)}}{1-\rho(h(\nu))\sqrt{h(\nu)}}}$ is the linear coefficient of the expansion of $g_{\nu}$ around its zero point $x_1^{\nu}$, and $C(0)$ the quadratic coefficient of the expansion of $g_0$ around its zero point $0$. The coefficients $c_i(\nu)$ are explicitely determined by the above equality and the coefficients of the expansion of $g_\nu$.

Inverting explicitely, we get
$$
P_\nu^{-1}(2\varepsilon)=\sqrt{r^2(h(\nu))+\frac{2\varepsilon}{C(0)}}-r(h(\nu))=\tilde\eta\left(\frac{2\varepsilon}{C(0)},r^2(h(\nu))\right),
$$
where $\tilde\eta$ is as defined before in Definition~\ref{def:pet}, and $r(\nu):=\frac{1-e^{-\frac{2\sqrt{\nu}}{1-\rho(\nu)\sqrt{\nu}}}}{2C(0)}.$ By \eqref{eq:mi}, 
\begin{equation}\label{eq:jj}\eta(2\varepsilon,\nu)=\theta_{-x_1^\nu}\circ g_\nu^{-1}(2\varepsilon)=\varphi_{\nu} \Big(\tilde\eta\big(\frac{2\varepsilon}{C(0)},r^2(h(\nu))\big)\Big),\end{equation} 
where $\chi_{\nu}:=\psi_{\nu}^{-1}\in\mathrm{Diff}_{\mathrm{id}}(\mathbb R,0)$ is a diffeomorphism tangent to the identity. 


\noindent Note that $\psi_\nu$ is analytic, since the above equality \eqref{eq:mi} can equivalently be written  as:
$$
P_{\nu}^{-1}\circ g_\nu\circ\theta_{x_1^\nu}=\psi_\nu.
$$
Now, $P_{\nu}^{-1}\circ g_\nu\circ \theta_{x_1^\nu}$ is an analytic germ at $0$, tangent to the identity, for every $\nu\in[0,\delta)$. Indeed, for $\nu=0$, $P_{0}^{-1}=\sqrt x$, and it follows by the binomial expansion, since $g_0$ is an analytic germ of multiplicity $2$. For $\nu>0$, $g_\nu$ is an analytic germ  tangent to the identity at $x_1^\nu$, and $P_\nu$ is an analytic diffeomorphism tangent to the identity at $0$, so the composition $P_{\nu}^{-1}\circ g_\nu\circ \theta_{x_1^\nu}$ is an analytic diffeomorphism at $0$ tangent to the identity. 
\end{proof}

\smallskip

\begin{lem}[Properties of the compensator $\kappa$]\label{prop:aux}
The following properties of the compensator $\kappa$ defined in \eqref{eq:kapa} hold:
\begin{itemize}
\item[$(i)$]
$$
\frac{d}{dx}\kappa^k=-k\kappa^{k+1},\ k\geq 1,
$$

\item[$(ii)$]
$$
\frac{d^{k}}{dx^k}\alpha(x,\nu)=P_{k+1}\left(\frac{1}{x},\kappa(x,\nu)\right),\ k\in\mathbb N_{\geq 1},
$$
where $P_k$ is a homogenous polynomial in two variables of degree $k$, with coefficients independent of $\nu$,
\item[$(iii)$] $$\frac{d}{dx}\log(x+\nu)=\kappa(x,\nu),\ \frac{d^{k+1}}{dx^{k+1}}\log(x+\nu)=(-1)^kk!\cdot\kappa(x,\nu)^{k+1},\ k\geq 1.$$
\end{itemize}
\end{lem}
Note that, for every homogenous polynomial $Q_k$ of degree $k\in\mathbb N_{\geq 1}$, it holds that $Q_{k}\left(\frac{1}{x},\kappa(x,\nu)\right)\to c_k\frac{1}{x^{k}}$ pointwise, as $\nu\to 0$, where $c_k\in\mathbb R$.

\begin{lem}\label{lem:pomo} Let $I(\nu,\eta)$ be as in \eqref{eq:iii} and let $k_\nu$, $C(\nu)$, $h(\nu)$ be as defined in Theorem~B. Let $\tilde\eta$ be as in Theorem~C. The following expansion in $\tilde \eta$ holds:
\begin{align*}
&I(h(\nu),k_\nu(\eta))=\alpha(C(\nu)\tilde\eta,2\sqrt{h(\nu)})+\frac{\rho(h(\nu))}{2}\Big(\log\tilde\eta+\log\left(C(\nu)\tilde\eta+2\sqrt{h(\nu)}\right)\Big)+\\
&\ \ \ \ \ \ \ \ \ +\sum_{k=0}^{\infty}N_k^\nu\left(\tilde\eta,\kappa\left(C(\nu)\tilde\eta,2\sqrt{h(\nu)}\right)\right).
\end{align*}
Here, $N_k^\nu$ are homogenous polynomials of degree $k$, whose coefficients depend on $\nu$, $k\geq 0$.
\end{lem}

\begin{proof} By \eqref{eq:iii},
\begin{align*}
I(h(\nu),k_\nu(\eta))=\alpha(k_\nu(\eta),2\sqrt{h(\nu)})+&\frac{\rho(h(\nu))}{2}\log \big(k_\nu^2(\eta)+\\
&\qquad +2\sqrt{h(\nu)}\cdot k_\nu(\eta)\big)-\Psi_{h(\nu)}^{mod}(x_0).
\end{align*}
Recall, from Theorem B, that $k_\nu(\eta)=C(\nu)\eta+o_\nu(\eta)$ is a real analytic germ of a diffeomorphism, analytic also in $\nu$, with asymptotic expansion as $\tilde\eta\to 0+$ in $\mathbb R\{\nu\}[[\tilde\eta]]$. On the other hand, by Lemma~\ref{prop:ik}, $\eta=\chi_\nu(\tilde\eta)$, where $\chi_\nu$  is a diffeomorphism tangent to the identity. Therefore, putting $K_\nu:=k_\nu\circ \chi_\nu$, we get
$$
k_\nu(\eta)=K_\nu(\tilde\eta),\ K_\nu=C(\nu)\cdot\mathrm{id}+h.o.t.,
$$
where $K_\nu$ is an analytic germ of a diffeomorphism, for every $\nu\in[0,\delta)$, i.e. with asymptotic expansion as $\tilde\eta\to 0+$ in $\mathbb R\{\nu\}[[\tilde\eta]]$.

We expand, using Lemma~\ref{prop:aux} and denoting by $\partial_1$ the partial derivative with respect to the first variable:
\begin{align*}
&\alpha(k_\nu(\eta),2\sqrt{h(\nu)})=\alpha(K_\nu(\tilde\eta),2\sqrt{h(\nu)})=\\
&\ =\alpha(C(\nu)\tilde\eta,2\sqrt{h(\nu)})+\partial_1\alpha(C(\nu)\tilde\eta,2\sqrt{h(\nu)})\cdot (K_\nu(\tilde\eta)-C(\nu)\tilde\eta)+\\
&\ \ \ +\frac{1}{2}\partial_1^2\alpha(C(\nu)\tilde\eta,2\sqrt{h(\nu)})\cdot (K_\nu(\tilde\eta)-C(\nu)\tilde\eta)^2+o_\nu((K_\nu(\tilde\eta)-C(\nu)\tilde\eta)^2)\\
&\ =\alpha(C(\nu)\tilde\eta,2\sqrt{h(\nu)})+\\
&\qquad\qquad\qquad+\sum_{k=1}^\infty P_{k+1}\left(\frac{1}{C(\nu)\tilde\eta},\kappa\left(C(\nu)\tilde\eta,2\sqrt{h(\nu)}\right)\right)\!\cdot\! (K_\nu(\tilde\eta)-C(\nu)\tilde\eta)^k=\\
&\ =\alpha(C(\nu)\tilde\eta,2\sqrt{h(\nu)})+\sum_{k=0}^{\infty}H_k^\nu\left(\tilde\eta,\kappa\left(C(\nu)\tilde\eta,2\sqrt{h(\nu)}\right)\right).
\end{align*}
Here, the coefficients of $P_k$ do not depend on $\nu$. The last line is obtained re-grouping the terms triangularly, where $H_k^\nu$ are homogenous polynomials of degree $k$ whose coefficients depend on $\nu$, $k\geq 0$.

Furthermore, by Taylor expansion of the logarithmic term and Lemma~\ref{prop:aux}, we get:
\begin{align*}
&\log \big(k_\nu^2(\eta)+2\sqrt{h(\nu)}\cdot k_\nu(\eta)\big)=\\
&\ =\log \big(K_\nu^2(\tilde\eta)+2\sqrt{h^{-1}(\nu)}\cdot K_\nu(\tilde\eta)\big)=\log\left(K_\nu(\tilde\eta)\right)+\log(K_\nu(\tilde\eta)+2\sqrt{h(\nu)})=\\
&\ =\log(\tilde\eta)+r_\nu(\tilde\eta)+\log\left(C(\nu)\tilde\eta+2\sqrt{h(\nu)}\right)+\\
&\ \ \ \ +\kappa\left(C(\nu)\tilde\eta,2\sqrt{h(\nu)}\right)\left(K_\nu(\tilde\eta)-C(\nu)\tilde\eta\right)+\\
&\ \ \ \  +\sum_{k=2}^{\infty}(-1)^{k-1}(k-1)!\cdot\kappa^{k}\left(C(\nu)\tilde\eta,2\sqrt{h(\nu)}\right)\cdot \left(K_\nu(\tilde\eta)-C(\nu)\tilde\eta\right)^k=\\
&\ =\log\tilde\eta+r_\nu(\tilde\eta)+\log\left(C(\nu)\tilde\eta+2\sqrt{h(\nu)}\right)+\sum_{k=3}^{\infty}M_k^\nu\left(\tilde\eta,\kappa\big(C(\nu)\tilde\eta,2\sqrt{h(\nu)}\big)\right).
\end{align*}
Here, $r_\nu$ is an analytic germ of diffeomorphism, with asymptotic expansion as $\tilde\eta\to 0+$ in $\mathbb R[[\tilde\eta]]$, for every $\nu\in[0,\delta)$. The coefficients of the expansion are analytic germs at $\nu=0$, as also is $C_\nu$. Also, $M_k^\nu$ are homogenous two-variable polynomials of degree $k$ with coefficients depending on $\nu$.
\end{proof}

\begin{proof}[Proof of Theorem~C] We use the expansion \eqref{eq:expa1}, from Theorem~B, the fact that $\eta=\tilde{\eta}+O_{\nu}(\tilde{\eta}^2)$, which follows by Lemma~\ref{prop:ik} and Lemma~\ref{lem:pomo}. The expansion follows after regrouping in a same block the terms (with their respective coefficients in $\nu$) that merge to the same asymptotic term in $\tilde\eta$ for $\nu=0$. Note that $c_2(0)\neq 0$, so $c_2(\nu)\neq 0$, for $\nu\in[0,\delta)$, by continuity. Therefore, all terms in expansion \eqref{eq:expa1} after $c_2(\nu)\cdot I(h(\nu),k_\nu(\eta))\eta^2$ can be factored through $c_2(\nu)$. 


\end{proof}

\subsection{Expansions in cases $\nu=0$ and $\nu>0$}\label{ss:prvi}
We now use the expansion \eqref{eq:ce} from Theorem~C to get expansions in $\varepsilon$, as $\varepsilon\to 0$. 

In the case $\nu=0$, $\tilde\eta=\sqrt{\frac{2\varepsilon}{C(0)}}$, and \eqref{eq:ce} immeditely becomes:
\begin{align*}
\ell^c(T_{\varepsilon,0})&\sim \frac{c_2(0)}{C(0)}\tilde\eta+\rho(0)\sum_{k=0}^\infty b_k(0)\tilde\eta^{k+2}\log\tilde\eta+\sum_{k=1}^{\infty}c_k\tilde\eta^{k+2}\\
&=\frac{c_2(0)\sqrt 2}{C(0)^{3/2}}\varepsilon^{\frac 1 2}+\rho(0)\sum_{k=0}^\infty c_k\varepsilon^{\frac{k+2}{2}}\log\varepsilon+\sum_{k=1}^{\infty}d_k\varepsilon^{\frac{k+2}{2}},\ \varepsilon\to 0,\ c_k,\,d_k\in\mathbb R. 
\end{align*}
Note that logarithmic terms appear only thanks to nontrivial residual invariant $\rho(0)$ of the parabolic time-one map for $\nu=0$.

\medskip
In the case $\nu>0$, under notation of Theorem~C, by \eqref{eq:dod}
$$
\tilde\eta=\frac{2\varepsilon}{C(0)\cdot r(h(\nu))}+o(\varepsilon)\in\mathbb R_\nu\{\varepsilon\}.
$$
Furthermore,
\begin{align}\label{eq:too}
\alpha(C(\nu)\tilde\eta,2\sqrt{h(\nu)})\sim&-\frac{1}{2\sqrt{h(\nu)}}\log\tilde\eta+\frac{1}{2\sqrt{h(\nu)}}\log\frac{2\sqrt{h(\nu)}}{C(\nu)}+\tilde\eta\mathbb R_{\nu}\{\tilde\eta\}\nonumber\\
\sim-\frac{1}{2\sqrt{h(\nu)}}&\log\varepsilon+\frac{1}{2\sqrt{h(\nu)}}\log\frac{C(0)r(h(\nu))\sqrt{h(\nu)}}{C(\nu)}+\varepsilon\mathbb R_{\nu}\{\varepsilon\},\nonumber\\
\log\left(\tilde\eta+\frac{2\sqrt{h(\nu)}}{C(\nu)}\right)\sim&\log\frac{2\sqrt{h(\nu)}}{C(\nu)}+\frac{C(\nu)}{2\sqrt{h(\nu)}}\tilde\eta+\tilde\eta^2\mathbb R_\nu\{\tilde\eta\}\nonumber\\
\sim&\log\frac{2\sqrt{h(\nu)}}{C(\nu)}+\frac{2C(\nu)}{C(0)r(h(\nu))\sqrt{h(\nu)}}\varepsilon+\varepsilon^2\mathbb R_\nu\{\varepsilon\},\nonumber\\
\kappa(C(\nu)\tilde\eta,2\sqrt{h(\nu)})\sim&\frac{1}{2\sqrt{h(\nu)}}-\frac{C(\nu)}{4h(\nu)}\tilde\eta+\tilde\eta^2\mathbb R_\nu\{\tilde\eta\}\nonumber\\
\sim&\frac{1}{2\sqrt{h(\nu)}}-\frac{C(\nu)}{2C(0)h(\nu)r(h(\nu))}\varepsilon+\varepsilon^2\mathbb R_\nu\{\varepsilon\}, \ \varepsilon\to 0.
\end{align}
Inserting \eqref{eq:too} in \eqref{eq:ce}, we see that there are no noninteger powers of $\varepsilon$ in the expansion, but there are additional logarithmic terms that are not related to non-zero residual invariant, coming from compensators. The monomials in the expansion are  $\varepsilon^k$, $k\in\mathbb N_0$, and $\varepsilon^k\log\varepsilon$, $k\geq 1$. 

More precisely, using the above calculations and proof of Corollary~\ref{cor:discrete} for the relation between $\ell(T_{\varepsilon,\nu})$ and $\ell^c(T_{\varepsilon,\nu})$:
\begin{align}\label{eq:too1}
\ell(T_{\varepsilon,\nu})=\bigg(\frac{1}{\sqrt{h(\nu)}}-\frac{\rho(h(\nu))}{C(0)}\bigg)\varepsilon (-\log\varepsilon)+
o(\varepsilon\log\varepsilon),\ \varepsilon\to 0.
\end{align}
Note that the first term $\varepsilon(-\log\varepsilon)$ in the case $\nu>0$ exists even if  the residual invariant $\rho(\nu)$ is zero. It is related to the hyperbolic nature of the orbit, as compared with the parabolic nature  in the case $\nu=0$, where logarithms appear only thanks to the nonzero residual term.
\begin{obs}[Reading the analytic invariant $\rho(\nu)$]\label{rem:zadnji} This paper was concerned with reading the analytic invariant $k$ (the multiplicity of $0$ in the unfolding) of a generic analytic unfolding of vector field \eqref{eq:polje},\ \eqref{eq:gene} with a saddle-node singular point, from the Chebyshev system expanding the length of the $\varepsilon$-neighborhoods of orbits in the unfolding. It is equal to the number of terms of this Chebyshev system that  \emph{disappear} at the bifurcation value $\nu=0$. 

Note that also the other invariant $\nu\mapsto\rho(\nu),\ \nu\geq 0,$  can be read from $\varepsilon$-neighborhoods of hyperbolic orbits in the unfolding. Indeed, under the assumption that the unfolding is already prenormalized by a constant homothecy $\varphi_{\nu}(x)=ax$, $\nu\geq 0$, we may assume that $\varphi_0'(0)=1$, that is, that $C(0)=1$. In that case, from the first term, expanded asymptotically as function of $\nu>0$, of the hyperbolic expansion as $\varepsilon\to 0$ of $\varepsilon\mapsto \mathcal \ell(T_{\varepsilon,\nu})$, we read $h(\nu)$ and then the other invariant 
$\rho(\nu)$ of the unfolding, as can be seen by the formula \eqref{eq:too1}:
$$
\ell(T_{\varepsilon,\nu})\sim \bigg(\frac{1}{\sqrt{h(\nu)}}-\rho(h(\nu))\bigg)\varepsilon (-\log\varepsilon),
\ \varepsilon\to 0,\ \nu>0.$$
Note that $h(\nu)\to 0$, as $\nu\to 0$, and that $\nu\mapsto h(\nu)$ and $\nu\mapsto\rho(\nu)$ are analytic germs at $\nu=0$ (since the unfolding is analytic in the variable and in the parameter, so the analytic normal form \eqref{eq:model} is analytic in $\nu$).  Let $h(\nu)=\sum_{i\geq 1}a_i \nu^i$, $a_i\in\mathbb R$, and $\rho(\nu)=\rho(0)+\sum_{i\geq 1}b_i \nu^i,$ $b_i\in\mathbb R,$ be their Taylor expansions. Note that $\rho(0)=\lim_{\nu\to 0}\rho(h(\nu))$ is the invariant of the  saddle-node point for bifurcation value $\nu=0$. Expanding the first coefficient of \eqref{eq:too1} in $\nu$, as $\nu\to 0$,
$$
\frac{1}{\sqrt{h(\nu)}}-\rho(h(\nu))
$$
we recover $a_i$ triangularly from coefficients of rational powers $k-\frac{1}{2}$, $k\in\mathbb N_0$, of $\nu$, and, simultaneously, $b_i$ from coefficients of integer powers of $\nu$. In other words, the coefficient of $\varepsilon(-\log\varepsilon)$, as function of $\nu>0$, can be decomposed in a unique way as a sum of an analytic function and a negative square root of an analytic function. The function under the negative square root is then $h(\nu)$ and $\rho(\nu)$ is the other analytic function postcomposed by $h^{-1}(\nu)$.

Note that this reconstruction of $h(\nu)$ and $\rho(\nu)$ relies heavily on the fact that normal forms are analytic in the parameter. On the other hand, just the multiplicity $k$ in the unfolding can be reconstructed in the same way for smooth unfoldings, if the  order of smoothness is bigger than the multiplicity \cite{KR}.
\end{obs}

\section{Concluding remark: the case $\nu<0$}\label{ss:drugi}\
Throughout this paper, we restricted the study of the unfolding \eqref{eq:polje} to parameter values $\nu\in[0,\delta)$. We explain here the reasons for this restriction.

For $\nu>0$, there are two real singular points. The orbit of $x_0>0$ \emph{accumulates} at one of these singular points. 
If $\nu<0$, there are no \emph{real} singular points and the real orbit passes near zero and goes through to $-\infty$ in a finite real time. 
However, it can be seen that, the smaller the $|\nu|$, the more and more iterations are defined on the real line (of order of growth $|\nu|^{-1/2}$). That is, as $\nu\to 0$, the  \emph{density} of iterates around $0$ increases. 
Therefore, for a small $\varepsilon>0$ and $|\nu|$ sufficiently small with respect to $\varepsilon$, it is possible to define the critical time, and, consequently, the \emph{tail}, as the union of $\varepsilon$-neighborhoods of finitely many first iterations from $x_0$, as long as the distance between the consecutive two points remains larger than or equal to $2\varepsilon$. However, the tail will thus be defined only in a  region of $(\varepsilon_{>0},\nu)$-half-plane. It is a full neighborhood of the origin in the first quadrant, but not in the fourth quadrant. 

Additional difficulty is that, unlike the case $\nu>0$, for $\nu<0$ there are no fixed points of the time-one map on the real line to which the orbits on the real line converge. Therefore, we do not have a \emph{good} choice of a point for the asymptotic expansions as in Proposition~\ref{prop:pm} and Theorem B. 

For $\nu\geq0$, we have calculated the critical time using the time coordinate. In order to have a uniform expansion, we need some continuity of the time coordinate with respect to $\nu$ (continuity of the domain of definition and of the function), that was obtained by use of the \emph{compensators}. 
For $\nu>0$, we use the time coordinate defined around one (right-most) hyperbolic singular point, which extends until the left singular point, where it ramifies. Its domain and the time coordinate itself converge to the 'global' time coordinate $\frac{1}{z}+\rho(0)\log z$ as $\nu\to 0$ (see also Glutsyuk \cite{G}.) 

For $\nu>0$ the fundamental domains are annuli around singular points on the real line. For $\nu<0$, the \emph{real} orbit of $x_0>0$ lies in the passage between the two (complex) indifferent singular points with rotational linear part.
Here, the natural space of orbits is a \emph{crescent-like} fundamental domain with the two tips at the two complex critical points. 
This approach was studied by Lavaurs in \cite{L} and resumed in \cite{MRouss}. In \cite{MRouss}, the authors precise the difference between the two charts which they call Lavaurs and Glutsyuk charts. 

Opening of the passage between the two singular points for $\nu<0$ gives a mapping between the two domains:
as each crescent-like fundamental domain, for $\nu<0$, corresponds holomorphically  (by passing to the quotient) to a Riemann sphere with two marked points, a global mapping between these crescents corresponds to a global mapping of $\mathbb{CP}^1$ preserving two points ($0$ and $\infty$). Hence, to a linear mapping determined by one number, which is called the \emph{Lavaurs period}. Note that the Lavaurs period does not have a limit as $\nu$ tends to zero.

The study of the critical time for $\nu<0$ may be related to the concept of Lavaurs period. 

\color{black}

\section{Declarations}
\subsection{Ethical Approval}
Not applicable.

\subsection{Funding}
All authors are  supported by the Croatian Science Foundation grants PZS-2019-02-3055 and IP-2022-10-9820, and all authors except R.H. by the  bilateral Hubert-Curien Cogito grants 2021-22 and 2023-2024. Maja Resman is supported by the Croatian Science Foundation grant no. UIP-2017-05-1020. Pavao Marde\v si\' c and Maja Resman also express their gratitude to the Fields Institute for supporting their stay in the scope of the \emph{Thematic Program on Tame Geometry, Transseries and Applications to Analysis and Geometry 2022}. Pavao Marde\v si\' c was also partially supported by  
EIPHI Graduate School (contract ANR-17-EURE-0002).

\bibliographystyle{plain}

\end{document}